\documentclass[a4paper]{amsart} 
\usepackage{amssymb} 
\usepackage[all]{xy}
\SelectTips{cm}{}

\usepackage{hyperref}
\calclayout

\title{Report on locally finite triangulated categories}

\author{Henning Krause}
\address{Henning Krause\\ Fakult\"at f\"ur Mathematik\\
Universit\"at Bielefeld\\ D-33501 Bielefeld\\ Germany.}
\email{hkrause@math.uni-bielefeld.de}

\newtheorem{lem}{Lemma}[section]
\newtheorem{prop}[lem]{Proposition}
\newtheorem{cor}[lem]{Corollary}
\newtheorem{thm}[lem]{Theorem}

\theoremstyle{remark}
\newtheorem{rem}[lem]{Remark}

\theoremstyle{definition}
\newtheorem{exm}[lem]{Example}

\numberwithin{equation}{section}

\renewcommand{\mod}{\operatorname{mod}\nolimits}
\DeclareMathOperator*{\colim}{colim}

\newcommand{\rad}{\operatorname{rad}\nolimits}
\newcommand{\Rad}{\operatorname{Rad}\nolimits}
\newcommand{\Irr}{\operatorname{Irr}\nolimits}
\newcommand{\add}{\operatorname{add}\nolimits}

\newcommand{\id}{\operatorname{id}\nolimits}

\newcommand{\Mod}{\operatorname{Mod}\nolimits}
\newcommand{\umod}{\operatorname{\underline{mod}}\nolimits}
\newcommand{\uEnd}{\operatorname{\underline{End}}\nolimits}
\newcommand{\End}{\operatorname{End}\nolimits}
\newcommand{\Hom}{\operatorname{Hom}\nolimits}
\newcommand{\RHom}{\operatorname{\mathbf{R}Hom}\nolimits}

\renewcommand{\Im}{\operatorname{Im}\nolimits}
\newcommand{\Ker}{\operatorname{Ker}\nolimits}
\newcommand{\Coker}{\operatorname{Coker}\nolimits}
\newcommand{\coh}{\operatorname{coh}\nolimits}

\renewcommand{\dim}{\operatorname{dim}\nolimits}

\newcommand{\Ext}{\operatorname{Ext}\nolimits}

\newcommand{\Tr}{\operatorname{Tr}\nolimits}

\newcommand{\uHom}{\operatorname{\underline{Hom}}\nolimits}

\newcommand{\Thick}{\operatorname{Thick}\nolimits}

\newcommand{\tor}{\operatorname{tor}\nolimits}
\newcommand{\cox}{\operatorname{cox}\nolimits}
\newcommand{\NC}{\operatorname{NC}\nolimits}
\newcommand{\MCM}{\operatorname{MCM}\nolimits}
\newcommand{\uMCM}{\operatorname{\underline{MCM}}\nolimits}
\newcommand{\Ab}{\mathsf{Ab}}
\newcommand{\op}{\mathrm{op}}

\newcommand{\inc}{\mathrm{inc}}
\newcommand{\can}{\mathrm{can}}

\newcommand{\lto}{\longrightarrow}
\newcommand{\xto}{\xrightarrow}

\def\a{\alpha}
\def\b{\beta}
\def\e{\varepsilon}
\def\d{\delta}
\def\g{\gamma}
\def\p{\phi}

\def\s{\sigma}
\def\t{\tau}

\def\De{\varDelta}
\def\Ga{\varGamma}

\def\Si{\varSigma}

\def\Om{\varOmega}

\def\A{{\mathsf A}}

\def\C{{\mathsf C}}

\def\Sc{{\mathsf S}}
\def\X{{\mathsf X}}

\def\T{{\mathsf T}}
\def\U{{\mathsf U}}
\def\V{{\mathsf V}}

\def\Oc{{\mathcal O}}

\def\bfA{{\mathbf A}}
\def\bfD{{\mathbf D}}
\def\bfK{{\mathbf K}}

\def\bfT{{\mathbf T}}

\def\bbP{{\mathbb P}}
\def\bbR{{\mathbb R}}
\def\bbZ{{\mathbb Z}}

\def\frp{{\mathfrak p}}

\begin{document}

\begin{abstract}
The basic properties of locally finite triangulated categories are
discussed. The focus is on Auslander--Reiten theory and the lattice of
thick subcategories.
\end{abstract}

\keywords{Triangulated category, derived category, locally noetherian,
  locally finite, thick subcategory, Auslander--Reiten theory}

\subjclass[2010]{18E30(primary), 16E35, 16G70}
\maketitle
\setcounter{tocdepth}{1}
\tableofcontents

\section{Introduction}

This is a report on a particular class of triangulated categories. A
triangulated category $\T$ is said to be \emph{locally finite} if
every cohomological functor from $\T$ or its opposite category
$\T^\op$ into the category of abelian groups is a direct sum of
representable functors.

We present a number of basic results for such triangulated
categories. Some of these results seem to be new, but we include also
results which are variations or generalisations of known results. Thus
our aim is to provide the foundations for studying the locally finite
triangulated categories.

A basic tool for understanding a triangulated category $\T$ is the
collection of representable functors $\Hom_\T(-,X)\colon\T^\op\to\Ab$
where $X$ runs through the objects of $\T$. We show that $\T$ is
locally finite if and only if each representable functor is of finite
length. This property justifies the term `locally finite' which is due
to Xiao and Zhu \cite{XZ} in the triangulated context and goes
back to Gabriel \cite{Ga1962}.

An important thread in the study of locally finite triangulated
categories is the use of Auslander--Reiten theory. The principal idea
is to analyse for each object $X$ in $\T$ the radical filtration
\[\ldots\subseteq\Rad^{2}_\T(-,X)\subseteq
\Rad^{1}_\T(-,X)\subseteq\Rad^{0}_\T(-,X)=\Hom_\T(-,X)
\]
which is finite when $\T$ is locally finite. Some of this information
is encoded in the Auslander--Reiten quiver of $\T$ which can be
described fairly explicitly.

Another intriguing invariant of a triangulated category is the lattice
of thick subcategories. Assuming locally finiteness, we show that the
inclusion of each thick subcategory admits a left and a right
adjoint. In fact, the lattice has interesting symmetries and is even
finite if the category is finitely generated.

The results presented here are most complete when the category is
\emph{simply connected}, that is, the Auslander--Reiten quiver is
connected and contains no oriented cycle. For instance, the lattice of
thick subcategories is in this case isomorphic to the lattice of
non-crossing partitions associated to some diagram of Dynkin type.

Using covering theory, the study of a general locally finite
triangulated category can often be reduced to the simply connected
case. For this direction we refer to recent work of Amiot
\cite{Am2007} and K\"ohler \cite{Ko}.

\subsection*{Acknowledgements} 
It is a pleasure to thank numerous colleagues for interesting
discussions and helpful comments on the subject of this work. Let me
mention explicitly Apostolos Beligiannis, Otto Kerner, Claudia
K\"ohler, Shiping Liu, Claus Ringel, Jan \v{S}\v{t}ov\'i\v{c}ek, Hugh
Thomas, and Dieter Vossieck. I am also grateful to an anonymous
referee for many helpful suggestions for improving the exposition.

\section{Locally noetherian triangulated categories}

In this section, locally noetherian and locally finite triangulated
categories are introduced. We provide various characterisations and a
host of examples. Then we establish the existence of adjoints for
inclusions of thick subcategories.

Throughout this work let $\T$ denote a triangulated category with
suspension $\Si$.

\subsection*{The abelianisation of a triangulated category}
Following Freyd \cite[\S3]{Fr1966} and Verdier \cite[II.3]{V}, we
consider the \emph{abelianisation} $\bfA(\T)$ of a triangulated
category $\T$ which is by definition the abelian category consisting
of all additive functors $F\colon\T^\op\to\Ab$ into the category of
abelian groups that fit into an exact sequence
\[\Hom_\T(-,X)\lto\Hom_\T(-,Y)\lto F\lto 0.\]
The fully faithful Yoneda functor $H\colon\T\to \bfA(\T)$ taking an
object $X$ to the representable functor $\Hom_\T(-,X)$ is the
universal cohomological functor starting in $\T$, that is, each
cohomological functor $\T\to\A$ to an abelian category $\A$ factors
essentially uniquely through $H$. Observe that the functor taking
$\Hom_\T(-,X)$ to $\Hom_\T(X,-)$ induces an equivalence
\begin{equation}\label{eq:dual}
\bfA(\T)^\op\stackrel{\sim}\lto\bfA(\T^\op).
\end{equation}
This is an immediate consequence of the universal property of the Yoneda functor.

\subsection*{Locally noetherian triangulated categories}
Given an essentially small triangulated category $\T$, we use its
abelianisation to formulate a useful finiteness condition; see also
\cite{Be2000}.  We say that $\T$ is \emph{locally
  noetherian}\footnote{The terminology refers to the equivalent fact
  that the abelian category of additive functors $\T^\op\to\Ab$ is
  locally noetherian in the sense of \cite[II.4]{Ga1962}.} if the
equivalent conditions of the following theorem are satisfied.

\begin{thm}\label{th:fin}
For an essentially small triangulated category $\T$ the following
conditions are equivalent.
\begin{enumerate}
\item Every cohomological functor $\T^\op\to\Ab$ into the category of
  abelian groups is a direct sum of representable functors.
\item Every object in $\T$ is a finite coproduct of indecomposable
objects with local endomorphism rings, and for every sequence
$X_1\xto{\p_1} X_2\xto{\p_2} X_3\xto{\p_3} \ldots$ of non-isomorphisms
between indecomposable objects there exists some number $n$ such that
$\p_n\ldots\p_2\p_1=0$.
\item Idempotents in $\T$ split and every object of the abelianization
  $\bfA(\T)$ is noetherian, that is,  every ascending chain of
  subobjects in  $\bfA(\T)$ eventually stabilises.
\end{enumerate}
\end{thm}
\begin{proof}
  We view the additive category $\T$ as a \emph{ring with several
    objects} and think of additive functors $\T^\op\to\Ab$ as
  $\T$-modules. Note that a $\T$-module is flat if and only if it is a
  cohomological functor; see \cite[Lemma~2.7]{Kra2000}.  Bass has
  characterised the rings for which every flat module is
  projective. This can be generalised to modules over rings with
  several objects, see \cite[Theorem~B.12]{JL}, and yields the
  equivalence of conditions (1) and (2).

Recall that a module $M$ is \emph{fp-injective} if $\Ext^1(-,M)$
vanishes on all finitely presented modules. Note that $\T$ as a ring
with several objects is \emph{noetherian} (that is, each representable
functor $\Hom_\T(-,X)$ satisfies the ascending chain condition on
subfunctors) if and only if every fp-injective $\T$-module is
injective; see \cite[Theorem~B.17]{JL}. The fp-injective $\T$-modules
are precisely the cohomological functors $\T^\op\to\Ab$, by
\cite[Lemma~2.7]{Kra2000}.

Suppose that (1) holds and fix an fp-injective $\T$-module $M$. Choose
an injective envelope $\p \colon M\to Q$. The snake lemma shows that
the cokernel $\Coker\p$ is cohomological. Thus $\Coker\p$ is a direct
sum of representable functors and therefore projective; in particular
$\p$ splits. It follows that $M$ is injective, and therefore $\T$ is
noetherian. It remains to show that $\T$ is idempotent complete. But
this is clear because a direct summand of a representable functor is
cohomological and therefore representable.  Thus (3) holds.

Now suppose that (3) holds.  Thus the ascending chain condition holds
for chains of finitely presented submodules of modules of the form
$\Hom_\T(-,X)$. This implies the ascending chain conditions for
arbitrary submodules, since each submodule is a union of finitely
generated submodules and each finitely generated submodule of a
finitely presented one is again finitely presented. It follows that
$\T$ is noetherian.

Fix a flat $\T$-module $M$ and choose an epimorphism $\p\colon P\to M$
such that $P$ is projective. The snake lemma shows that the kernel
$\Ker\p$ is cohomological. Thus $\Ker \p$ is fp-injective and
therefore injective; in particular $\p$ splits. It follows that $M$ is
projective, and therefore a direct sum of finitely generated
projective modules; see \cite[Corollary~B.13]{JL}. The finitely
generated projective modules are precisely the representable functors
since $\T$ has split idempotents. Thus (1) holds.
\end{proof}

\begin{exm}\label{ex:noeth}
Fix a field $k$ and denote by $\A$ the category of $k$-linear
representations of the quiver
\[\Ga\colon\quad 1\stackrel{}\lto 2\stackrel{}\lto 3
\stackrel{}\lto 4\stackrel{}\lto \cdots\] that are finite dimensional
and have finite support.  Then the bounded derived category
$\bfD^b(\A)$ is locally noetherian but its opposite category is not.

Indeed, each object in $\bfD^b(\A)$ decomposes into a finite coproduct
of indecomposable objects with local endomorphism rings. The
indecomposable objects are isomorphic to complexes concentrated in a
single degree (thus of the form $X[i]$ with $X\in\A$
and $i\in\bbZ$) since $\Ext_\A^p(-,-)=0$ for $p>1$.

Given a sequence $X_1,\ldots,X_r$ of objects in $\A$ and a sequence of
morphism $X_1[i_1]\to X_2[i_2]\to \cdots\to X_r[i_r]$ in $\bfD^b(\A)$
such that their composite is non-zero, we have $i_1\le\ldots\le i_r\le
i_1+1$.  Thus it remains to observe that infinite chains of
non-isomorphisms between indecomposable objects in $\A$ exist only in
one direction. More precisely, for each $n\ge 1$ let $\A_n$ denote the
full subcategory of representations with support contained in
$\{1,\ldots,n\}$.  Then any non-zero morphism $X\to Y$ between
indecomposable representations has the property that $X\in\A_n$
implies $Y\in\A_n$. This is an immediate consequence of the fact that
each indecomposable representation of $\Ga$ is, up to isomorphism, of
the following form:
\[0\to\cdots \to 0\to k\xto{1}\cdots\xto{1}k\to 0\to\cdots\]

On the other hand, there is an obvious chain of
proper epimorphisms $\cdots \to X_3\to X_2\to X_1$ in $\A$, where
$X_n$ denotes the unique indecomposable representation with support
$\{1,\ldots,n\}$.
\end{exm}

\subsection*{Locally finite triangulated categories}
An essentially small triangulated category $\T$ is said to be
\emph{locally finite} if $\T$ and $\T^\op$ are locally noetherian.
The first condition means that each representable functor
$\Hom_\T(-,X)$ is a noetherian $\T$-module. In particular, each
subobject belongs to $\bfA(\T)$.  Combining the second condition with
the equivalence $\bfA(\T)^\op\stackrel{\sim}\to\bfA(\T^\op)$, it
follows that $\Hom_\T(-,X)$ is an artinian $\T$-module. Thus locally
finite means that each representable functor $\Hom_\T(-,X)$ is of
finite length as a $\T$-module. In particular, $\T$ is locally finite
if and only if the category of $\T$-modules is locally finite in the
sense of \cite[II.4]{Ga1962}.

Suppose that $\T$ is locally finite and fix an object $Y$. Then there
are only finitely many isomorphism classes of indecomposable objects
$X$ satisfying $\Hom_\T(X,Y)\neq 0$, because $\Hom_\T(X,Y)\neq 0$
implies that $\Hom_\T(-,X)$ is the projective cover of a composition
factor of $\Hom_\T(-,Y)$. This observation gives rise to the following
characterisation which can be deduced from
\cite[Theorem~2.12]{Au1974}.

\begin{prop}[Auslander]\label{pr:fin}
\pushQED{\qed} An essentially small triangulated category $\T$ with
split idempotents is locally finite if and only if for each object $Y$
the following holds:
\begin{enumerate}
\item The object $Y$ decomposes into a finite direct sum of
  indecomposable objects.

\item There are only finitely many isomorphism classes of
  indecomposable objects $X$ satisfying $\Hom_\T(X,Y)\neq 0$.
\item For each indecomposable object $X$, the $\End_\T(X)$-module
  $\Hom_\T(X,Y)$ is of finite length.\qedhere
\end{enumerate}
\end{prop}

\subsection*{Examples}
We list some examples of triangulated categories that are locally
finite.  Throughout we fix a field $k$.

(1) Let $\T$ be an essentially small $k$-linear triangulated
category. Suppose that idempotents split and that morphism spaces are
finite dimensional.  Then $\T$ is locally finite if and only if for
each object $Y$ there are only finitely many isomorphism classes of
indecomposable objects $X$ satisfying $\Hom_\T(X,Y)\neq 0$. This
follows from Proposition~\ref{pr:fin} and serves as a definition in
\cite{XZ}.

(2) Let $A$ be a finite dimensional $k$-algebra and suppose that $k$
is algebraically closed. Then the bounded derived category
$\bfD^b(\mod A)$ of the category of finite dimensional $A$-modules is
locally finite if and only if it is triangle equivalent to
$\bfD^b(\mod k\Ga)$ for some path algebra $k\Ga$ of a finite quiver
$\Ga$ such that  its underlying diagram is a disjoint union of
diagrams of Dynkin type; see \cite[\S5]{Ha1988} and \cite[Theorem~12.20]{Be2000b}.

(3) Let $\A$ be an essentially small hereditary abelian category. Then
the derived category $\bfD^b(\A)$ is locally finite if and only if
$\A$ satisfies the conditions in Proposition~\ref{pr:fin}.  This
follows from the fact that each indecomposable object is isomorphic to
a complex that is concentrated in a single degree.  If $\A$ is the
category of finitely generated modules over an artinian ring, then
this condition means that the ring is of finite representation type.

(4) Let $A$ be a noetherian ring and suppose that $A$ is
\emph{Gorenstein}, that is, $A$ has finite injective dimension as an
$A$-module. Denote by $\MCM(A)$ the category of finitely generated
$A$-modules $X$ that are \emph{maximal Cohen--Macaulay}, which means
that $\Ext_A^i(X,A)=0$ for all $i>0$. This is an exact Frobenius
category, and the stable category $\uMCM(A)$ modulo all morphisms that
factor through a projective object is a triangulated category
\cite{Bu}.

If $A$ is a finite dimensional and self-injective $k$-algebra, then
all $A$-modules are maximal Cohen--Macaulay, and $\uMCM(A)$ is locally
finite if and only if $A$ is of \emph{finite representation type},
that is, there are only finitely many isomorphism classes of
indecomposable $A$-modules. These algebras have been classified
\cite{Ri1983,Wa}.

If $A$ is a commutative complete local ring, then $A$ is by definition
of \emph{finite Cohen--Macaulay type} if there there are only finitely
many isomorphism classes of indecomposable maximal Cohen--Macaulay
modules over $A$.  In that case $\uMCM (A)$ is locally finite.  There
is a whole theory describing such rings and a parallel theory for
graded Gorenstein algebras; see \cite{Yo}.

(5) Let $\Ga$ be a quiver of Dynkin type. Then the orbit category
$\bfD^b(\mod k\Ga)/G$ of the derived category with respect to an
appropriate group $G$ of automorphisms is a locally finite
triangulated category \cite{Ke2005,Am2007}. Examples are the cluster
categories of finite type \cite{BMRRT}.

(6) The category of finitely generated projective modules over the
ring $\bbZ/4\bbZ$ carries a triangulated structure that admits no
model \cite{MSS}; it is a locally finite triangulated category.

\subsection*{Orthogonal subcategories}

Let $\T$ be a triangulated category and $\Sc$ a triangulated
subcategory. Then we define two full subcategories
\begin{align*}
\Sc^\perp&=\{Y\in\T\mid\Hom_\T(X,Y)=0\text{ for all }X\in\Sc\}\\
^\perp\Sc&=\{X\in\T\mid\Hom_\T(X,Y)=0\text{ for all }Y\in\Sc\}
\end{align*}
and call them \emph{orthogonal subcategories} with respect to $\Sc$.
Note that $\Sc^\perp$ and $^\perp\Sc$ are thick subcategories of $\T$.

The following lemma collects some basic facts about orthogonal
subcategories which are well-known. For a proof, see \cite[Proposition~4.9.1]{Kr2010}.

\begin{lem}\label{le:tria-loc}
\pushQED{\qed}
Let $\T$ be a triangulated category and $\Sc$ a thick subcategory. Then
the following are equivalent.
\begin{enumerate}
\item The inclusion functor $\Sc\to\T$ admits a right adjoint.
\item The composite $\Sc^\perp\xto{\inc}\T\xto{\can}\T/\Sc$ is an equivalence.
\item The inclusion functor $\Sc^\perp\to\T$ admits a left adjoint and
  $^\perp(\Sc^\perp)=\Sc$.\qedhere
\end{enumerate}
\end{lem}

There is an interesting consequence. If $\Sc$ is a thick subcategory of $\T$ such
that the inclusion admits a left and a right adjoint, then one has
equivalences
\[^\perp\Sc\stackrel{\sim}\lto\T/\Sc\stackrel{\sim}\longleftarrow\Sc^\perp.\]

\subsection*{Existence of adjoints}

Triangulated categories that are locally noetherian have the following
remarkable property.

\begin{thm}\label{th:adj}
Let $\T$ be an essentially small triangulated category and suppose
that $\T$ is locally noetherian. Then for each thick subcategory of
$\T$ the inclusion functor admits a right adjoint.
\end{thm}
\begin{proof}
Fix a thick subcategory $\U$ and an object $X$ in $\T$. We need to
construct a morphism $U\to X$ with $U$ in $\U$ inducing a bijection
$\Hom_\T(U',U)\to\Hom_\T(U',X)$ for all $U'$ in $\U$. Take the comma
category $\U/X$ consisting of all morphisms $U\stackrel{\p}\to X$ with
$U$ in $\U$. A morphism from $U\stackrel{\p}\to X$ to
$U'\stackrel{\p'}\to X$ is a morphism $\mu\colon U\to U'$ such that
$\p'\mu=\p$.  This category is closely related to Verdier's
construction of the localisation functor $\T\to\T/\U$; see
\cite[II.2]{V}. The arguments given there show that $\U/X$ is
filtered. Therefore the functor
\[\colim_{U\to X}\Hom_\T(-,U)\] is cohomological.  
Moreover, one obtains an exact sequence of cohomological functors
\begin{multline*}
\cdots\lto \Hom_{\T/\U}(-,\Si^{-1}X)\lto \colim_{U\to
  X}\Hom_\T(-,U)\lto \\ \lto\Hom_\T(-,X)\lto \Hom_{\T/\U}(-,X)\lto
\cdots
\end{multline*}
since one has by definition
\[\Hom_{\T/\U}(-,X)=\colim_{X\to V}\Hom_\T(-,V)
\]
where $X\to V$ runs through all morphisms with cone in $\U$.  

Now we use that $\T$ is locally noetherian and write \[\colim_{U\to
  X}\Hom_\T(-,U)=\bigoplus_{i\in I}\Hom_\T(-,U_i)\] as a direct sum of
representable functors.  Similarly, we
get \[\Hom_{\T/\U}(-,X)=\bigoplus_{j\in J}\Hom_\T(-,V_j).\] We may
assume that $U_i$ and $V_j$ are non-zero for all $i,j$.  Observe that
$U_i\in \U$ and $V_j\in\U^\perp$ for all $i,j$.  The morphism
\[\Hom_\T(-,X)\lto\bigoplus_{j\in J}\Hom_\T(-,V_j)\] factors through a
finite sum $\bigoplus_{j\in J_0}\Hom_\T(-,V_j)$. In fact, the
exactness of the above sequence implies that $V_j$ belongs to $\U$ for
each $j\in J\smallsetminus J_0$. Thus $J=J_0$ and therefore
$\Hom_{\T/\U}(-,X)$ belongs to $\bfA(\T)$. It follows that $I$ is also
finite, since $ \colim_{U\to X}\Hom_\T(-,U)$ is an extension of two
objects in $\bfA(\T)$. This yields a morphism $U=\coprod_iU_i\to X$
inducing a bijection $\Hom_\T(U',U)\to\Hom_\T(U',X)$ for all $U'$ in
$\U$.
\end{proof}

\begin{cor}\label{co:perp}
Let $\T$ be an essentially small triangulated category and suppose
that $\T$ is locally noetherian.  If $\U$ is a thick subcategory of
$\T$, then $^\perp(\U^\perp)=\U$.
\end{cor}
\begin{proof}
  One could deduce this from Lemma~\ref{le:tria-loc}, but we give the
  complete argument because it is short and simple.  Clearly,
  ${^\perp(\U^\perp)}$ contains $\U$. Now pick an object $X$ in
  $^\perp(\U^\perp)$. Let $U\to X$ be the universal morphism from an
  object in $\U$ to $X$, and complete this to an exact triangle $U\to
  X\to V\to\Si U$. Then $V$ belongs to $\U^\perp$ and therefore
  $\Hom_\T(X,V)=0$. It follows that $X$ belongs to $\U$.
\end{proof}

The following example shows that the identity $^\perp(\U^\perp)=\U$
does not hold in general.

\begin{exm} Consider the bounded derived category $\bfD^b(\mod\bbZ)$ of
finitely generated modules over the ring $\bbZ$ of integers.  The
complexes with torsion cohomology form a thick subcategory
$\U=\bfD^b_{\tor}(\mod\bbZ)$ such that $^\perp\U=0$.
\end{exm}

\begin{cor}\label{co:quot}
Let $\T$ be an essentially small triangulated category and suppose
that $\T$ is locally noetherian. If $\U$ a thick subcategory of $\T$,
then the categories $\U$ and $\T/\U$ are locally noetherian.
\end{cor}
\begin{proof}
The inclusion $\U\to\T$ induces a fully faithful and exact functor
$\bfA(\U)\to\bfA(\T)$. Thus every object of $\bfA(\U)$ is
noetherian. On the other hand, there is an equivalence
$\U^\perp\xto{\sim}\T/\U$ by Lemma~\ref{le:tria-loc}. Thus $\T/\U$ is
locally noetherian, since $\U^\perp$ is locally noetherian.
\end{proof}

\section{Auslander--Reiten theory for triangulated categories}

We describe briefly the Auslander--Reiten theory for an essentially
small triangulated category $\T$ that is locally noetherian. For
general concepts from Auslander--Reiten theory we refer to
Appendix~\ref{se:appendix}.

\subsection*{Auslander--Reiten triangles}
Recall from \cite{Ha1987} that an exact triangle
\[X\stackrel{\a}\lto Y\stackrel{\b}\lto Z\stackrel{\g}\lto\Si X\] is
an \emph{Auslander--Reiten triangle} starting at $X$ and ending at $Z$
if $\a$ is a left almost split morphism and $\b$ is a right almost
split morphism.  Observe that these properties hold if and only if the
morphism $\b$ is minimal right almost split, by
\cite[Lemma~2.6]{Kr2000b}.

\begin{prop}\label{pr:ARtriangle}
Given an essentially small triangulated category that is locally
noetherian, there exists for each indecomposable object an
Auslander--Reiten triangle ending at it.
\end{prop}
\begin{proof}
Fix an indecomposable object $Z$ in $\T$.  Then the simple $\T$-module
$S_Z=\Hom_\T(-,Z)/\Rad_\T(-,Z)$ is finitely presented since $\T$ is
locally noetherian. Thus we can choose in $\bfA(\T)$ a minimal
projective presentation
\[\Hom_\T(-,Y)\lto\Hom_\T(-,Z)\lto S_Z\lto 0;\]
see Proposition~\ref{pr:KRS}.  It follows from Lemma~\ref{le:cover} that the
induced morphism $Y\to Z$ is minimal right almost split.  Completing
this morphism to an exact triangle yields an Auslander--Reiten
triangle ending at $Z$.
\end{proof}

The definition of an Auslander--Reiten triangle is symmetric. Thus
there are Auslander--Reiten triangles in $\T$ starting at each
indecomposable object if $\T^\op$ is locally noetherian. This gives
the existence of Auslander--Reiten triangles for locally finite
triangulated categories.  For compactly generated triangulated
categories, this result is due to Beligiannis
\cite[Theorem~10.2]{Be2004}.

\begin{cor}\label{co:AR}
\pushQED{\qed} Given an essentially small triangulated category that
is locally finite, there exist Auslander--Reiten triangles starting
and ending at each indecomposable object.\qedhere
\end{cor}

\begin{rem}\label{re:tau}
Let $X\to Y\to Z\to \Si X$ be an Auslander--Reiten triangle in an
essentially small triangulated category that is locally
noetherian. The relation between the end terms can be explained as
follows. Let $A=\End_\T(Z)$ and denote by $E=E(A/\rad A)$ an injective
envelope. Then
\begin{equation}\label{eq:tau}
  \Hom_A(\Hom_\T(Z,-),E)\cong\Hom_\T(-,\Si
  X);
\end{equation}
see \cite[Theorem~2.2]{Kr2000b}. In particular,
\begin{align*}
  \End_\T(X)&\cong \End_\T(\Si X)\\
  &\cong\Hom_A(\Hom_\T(Z,\Si X),E)\\
  &\cong\Hom_A(\Hom_A(\Hom_\T(Z,Z),E),E)\\
&\cong\End_A(E).
\end{align*}
\end{rem}

\subsection*{The Auslander--Reiten quiver}
For a locally finite triangulated category the structure of the
Auslander--Reiten quiver has been determined in work of Xiao and Zhu
\cite{XZ} and Amiot \cite{Am2007}. In fact, these authors consider
triangulated categories that are linear over a field with finite
dimensional morphism spaces. Then they apply the structural results on
valued translation quivers due to Riedtmann \cite{Ri1980} and Happel,
Preiser, and Ringel \cite{HPRa}.  The same arguments work in a
slightly more general setting, thanks to the following result; see also
\cite[Proposition~2.1]{Li}.

For the definition of a \emph{valued translation quiver}, see
\cite[\S2]{Li}. The original definition \cite[\S2]{HPRb} excludes
loops, but they are possible in our setting.

\begin{prop}\label{pr:AR}
Let $k$ be a commutative ring and $\T$ an essentially small $k$-linear
triangulated category such that all morphism spaces are of finite
length over $k$. Suppose that $\T$ is locally finite.  Then the
Auslander--Reiten quiver of $\T$ is a valued translation
quiver. Assigning to a vertex $X$ the length $\ell(X)$ of
$\Hom_\T(-,X)$ in the abelianisation $\bfA(\T)$ yields a subadditive
function on the set of vertices such that for each vertex $Z$
\[2\ell(Z)=\ell(Z)+\ell(\t Z)=2+\sum_{Y\to Z}\d_{Y,Z}\ell(Y).\]
\end{prop}
\begin{proof}
The existence of Auslander--Reiten triangles has already been
established, and this gives the translation $\t$. The identities for
the valuation are precisely the statements of Lemmas~\ref{le:val} and
\ref{le:val'}. For the second part one uses the fact that each
Auslander--Reiten triangle $\t Z\to \bar Y\to Z\to\Si (\t Z)$ induces
an exact sequence
\[0\to S_{\Si^{-1} Z}\to\Hom_\T(-,\t Z)\to\Hom_\T(-,\bar
Y)\to\Hom_\T(-,Z)\to S_Z\to 0\] in $\bfA(\T)$ by
Lemma~\ref{le:cover}. Then one applies Lemma~\ref{le:AR} which gives a
decomposition \[\bar Y=\coprod_{Y\to Z}Y^{\d_{Y,Z}}\] where $Y\to Z$
runs through all arrows ending at $Z$. It remains to observe that
$\ell(Z)=\ell(\t Z)$. This follows from the isomorphism \eqref{eq:tau}
and the alternative description of $\ell$ via
\[\ell(X)=\sum_C\ell_{\End_\T(C)}\Hom_\T(C,X),\] where $C$ runs through
a representative set of indecomposable objects; see
\cite[Proposition~2.11]{Au1974}.  Note that $\Hom_\T(-,X)$ and
$\Hom_\T(X,-)$ have the same length, thanks to the duality
\eqref{eq:dual}.
\end{proof}

\begin{thm}[Xiao--Zhu]\label{th:AR}
Let $k$ be a commutative ring and $\T$ an essentially small $k$-linear
triangulated category such that all morphism spaces are of finite
length over $k$.  Suppose that $\T$ is locally finite. Then each
connected component of the Auslander--Reiten quiver of $\T$ is of the
form $\bbZ\De/G$ for some tree $\De$ of Dynkin type and some group $G$
of automorphisms of $\bbZ\De$.
\end{thm}
\begin{proof}
Adapt the proofs of \cite[Theorem~2.3.4]{XZ} or
\cite[Theorem~4.0.4]{Am2007}, using Proposition~\ref{pr:AR}.
\end{proof}

\begin{exm}
Let $k$ be a field and $\Ga$ be a quiver of Dynkin type. Then the
Auslander--Reiten quiver of $\bfD^b(\mod k\Ga)$ is of the form
$\bbZ\Ga$; see \cite[\S4]{Ha1987}.
\end{exm}

A \emph{decomposition} $\C=\C_1\times\C_2$ of an additive category
$\C$ is a pair of full additive subcategories $\C_1$ and $\C_2$ such
that each object in $\C$ is a direct sum of two objects from $\C_1$
and $\C_2$, and $\Hom_\C(X_1,X_2)=0=\Hom_\C(X_2,X_1)$ for all
$X_1\in\C_1$ and $X_2\in\C_2$.  An additive category $\C$ is
\emph{connected} if any decomposition $\C=\C_1\times\C_2$ implies
$\C=\C_1$ or $\C=\C_2$.

\begin{prop}\label{pr:connect}
Let $\T$ be an essentially small triangulated category that is locally
finite. Then any non-zero morphism $X\to Y$ between two indecomposable
objects in $\T$ gives a path $X\to \cdots\to Y$ in the
Auslander--Reiten quiver of $\T$. In particular, the category $\T$ is
connected if and only if its Auslander--Reiten quiver is connected.
\end{prop}
\begin{proof}
We prove the first statement; the second statement is then an
immediate consequence. Let $\p\colon X\to Y$ be a non-zero
morphism. If $\p$ is invertible, then the path between $X$ and $Y$ has
length zero. Otherwise, $\p$ factors through the right almost split
morphism ending at $Y_0=Y$, which exists by
Proposition~\ref{pr:ARtriangle}. It follows from Lemma~\ref{le:AR}
that there is an arrow $Y_1\to Y_0$ and a non-zero morphism $X\to Y_0$
that factors through an irreducible morphism $Y_1\to Y_0$. We continue
with the corresponding morphism $X\to Y_1$, and the process terminates
since $\T$ is locally finite.
\end{proof}

It is interesting to note that the Auslander--Reiten quiver of $\T$
can be identified with the Ext-quiver \cite[7.1]{Ga1973} of the
abelian length category $\bfA(\T)$, using the bijection from
Lemma~\ref{le:simples} between the indecomposable objects of $\T$ and
the simple objects of $\bfA(\T)$. For the bijection between arrows one
uses Lemma~\ref{le:irr}.

\subsection*{The Nakayama functor}
Let $k$ be a field and $\T$ an essentially small $k$-linear
triangulated category with finite dimensional morphism spaces. Suppose
that $\T$ is locally finite.  Then there is for each object $X$ in
$\T$ an object $NX$ representing the $k$-dual of $\Hom_\T(X,-)$, that
is,
\[\Hom_k(\Hom_\T(X,-),k)\cong\Hom_\T(-,NX).\]
More precisely, we have an isomorphism
\[\Hom_k(\Hom_\T(X,-),k)\cong\bigoplus_{i\in I}\Hom_\T(-,Y_i)\]
for some collection of indecomposable objects $Y_i$ since $\T$ is
locally noetherian. It follows from Proposition~\ref{pr:fin} that $I$
is finite. Thus $NX=\coprod_iY_i$ is a representing object.  This
gives a functor $N\colon\T\to\T$ which is known as \emph{Nakayama
  functor} in representation theory \cite{Ga1980}, or as \emph{Serre
  functor} in algebraic geometry \cite{BK}.  It is easily checked that
this functor is an equivalence; a quasi-inverse is given by the
Nakayama functor for $\T^\op$ which sends an object $X$ to the object
representing $\Hom_k(\Hom_\T(-,X),k)$. The exactness then follows from
\cite[Proposition~3.3]{BK} or \cite[Theorem~A.4.4]{vdB}.

Note that for each indecomposable object $Z$ in $\T$, one obtains an
Auslander--Reiten triangle $X\to Y\to Z\to\Si X$ by first choosing a
non-zero $k$-linear map $\End_\T(Z)\to k$ annihilating the unique
maximal ideal of $\End_\T(Z)$ and then completing the corresponding
morphism $Z\to NZ$ to an exact triangle $\Si^{-1}(NZ)\to Y\to Z\to
NZ$; see \cite{Kr2000b} for details.

\section{The lattice of thick subcategories}

Let $\T$ be an essentially small triangulated category. We denote by $\mathbf T
(\T)$ the set of all thick subcategories of $\T$. This set is
partially ordered by inclusion. Observe that for any collection of
thick subcategories $\U_i$ its intersection $\bigcap_i\U_i$ is
thick. Thus $\bfT(\T)$ is a \emph{lattice}, that is, we can form for
each pair of thick subcategories $\U,\V$ its infimum
$\U\wedge\V=\U\cap\V$ and its supremum $\U\vee\V=\bigcap_{\Sc}\Sc$,
where $\Sc$ runs through all thick subcategories containing $\U$ and
$\V$.  In fact, the lattice $\bfT(\T)$ is \emph{complete} since for
any collection of thick subcategories $\U_i$ its infimum $\bigwedge_i
\U_i$ and its supremum $\bigvee_i \U_i$ exist.

If $X$ is an object or collection of objects in $\T$, we write
$\Thick(X)$ for the thick subcategory generated by $X$, that is, the
smallest thick subcategory containing $\T$. The triangulated category
$\T$ is \emph{finitely generated} if there is some
object $X$ such that  $\T=\Thick(X)$.

\subsection*{Compactness}
Let $L$ be a lattice. An element $x$ of $L$ is \emph{compact} if
$x\le\bigvee_{i\in I}y_i$ implies $x\le\bigvee_{i\in J}y_i$ for some
finite subset $J\subseteq I$. The lattice $L$ is \emph{compact} if it has a
greatest element that is compact.

\begin{lem}
\pushQED{\qed} The lattice $\bfT(\T)$ is compact if and only if $\T$
is finitely generated.  \qedhere
\end{lem}

\subsection*{Noetherianess}
A lattice $L$ is \emph{noetherian} if there is no infinite strictly
increasing chain $x_0<x_1<x_2<\ldots$ in $L$.

\begin{prop}
Let $\T$ be an essentially small triangulated category that is
finitely generated and locally noetherian. Then the lattice $\bfT(\T)$
is noetherian.
\end{prop}
\begin{proof}
  Let $\T=\Thick(X)$ and write $H=\Hom_\T(-,X)$. For each thick
  subcategory $\U$ let $\p_\U\colon X_\U\to X$ be the universal
  morphism from an object in $\U$ ending at $X$, which exists by
  Theorem~\ref{th:adj}. Note that $\U=\Thick(X_\U)$ since
  $\T=\Thick(X)$.  Denote by $H_\U$ the image of the induced morphism
  $\Hom_\T(-,X_\U)\to H$ in $\bfA(\T)$, and observe that the induced
  morphism $\pi_\U\colon \Hom_\T(-,X_\U)\to H_\U$ is a projective
  cover. This follows from the uniqueness of $\p_\U$ with
  Lemma~\ref{le:procov}, since any endomorphism $\Hom_\T(-,X_\U)\to
  \Hom_\T(-,X_\U)$ commuting with $\pi_\U$ is an isomorphism.

Let $\V$ be another thick subcategory of $\T$. Then $H_\U=H_\V$
implies $\U=\V$. Indeed, $H_\U=H_\V$ implies that their projective
covers are isomorphic. Thus $X_\U\cong X_\V$, and therefore
\[\U=\Thick(X_\U)=\Thick(X_\V)=\V.\]
Clearly, $\U\subseteq\V$ implies $H_\U\subseteq H_\V$. It follows that
$\bfT(\T)$ is noetherian, since the lattice of subobjects of $H$ in
$\bfA(\T)$ is noetherian.
\end{proof}

\begin{exm} 
Recall from Example~\ref{ex:noeth} that representations of the quiver
\[1\lto 2\lto 3\lto 4\lto\cdots\]
give rise to a triangulated category $\bfD^b(\A)$ that is locally
noetherian. The full subcategories $\A_n\subseteq\A$ consisting of all
representations with support contained in $\{1,\ldots,n\}$ yield an
infinite strictly increasing chain
$\bfD^b(\A_1)\subseteq\bfD^b(\A_2)\subseteq\ldots$ of thick
subcategories in $\bfD^b(\A)$.
\end{exm}

\subsection*{Complements}

Assigning to a thick subcategory $\U$ its orthogonal subcategories
$\U^\perp$ and $^\perp\U$ yields two order reversing maps
$\bfT(\T)\to\bfT(\T)$.  These maps are of interest because we have
for two objects $X,Y$ in $\T$
\[\Hom_\T^*(X,Y)=0 \quad\Longleftrightarrow \quad\Thick(X)^\perp\ge\Thick(Y) \]
where
\[\Hom_\T^*(X,Y)=\bigoplus_{n\in\bbZ}\Hom_\T(X,\Si^nY).\]
Let us collect the basic properties of both maps when $\T$
is locally finite.

\begin{prop}
Let $\T$ be an essentially small triangulated category and suppose
that $\T$ is locally finite.
\begin{enumerate}
\item The maps $\bfT(\T)\to\bfT(\T)$ taking $\U$ to $\U^\perp$ and
  $^\perp\U$ are mutually inverse. Thus the lattice $\bfT(\T)$ is
  self-dual.
\item Let $\U$ be a thick subcategory. Then $\U\wedge\U^\perp=0$ and
  $\U\vee\U^\perp=\T$.
\item Let $\V\subseteq\U\subseteq\T$ be thick subcategories.  Then the
  quotient $\U/\V$ is a locally finite triangulated category, and there
  is a lattice isomorphism \[ [\V,\U]=\{\X\in\bfT(\T)\mid\V\subseteq
 \X\subseteq\U\}\stackrel{\sim}\lto\bfT(\U/\V).\]
The map takes $\X\subseteq\U$ to $\X/\V$.
\end{enumerate}
\end{prop}
\begin{proof} 
(1) We have $^\perp(\U^\perp)=\U=(^\perp\U)^\perp$ by
  Corollary~\ref{co:perp}.

  (2) Clearly, $\U\cap\U^\perp=0$.  On the other hand, each object $X$
  in $\T$ fits into an exact triangle $U\to X\to V\to\Si U$ with
  $U\in\U$ and $V\in\U^\perp$, by Theorem~\ref{th:adj}. Thus
  $\U\vee\U^\perp=\T$.

  (3) The category $\U/\V$ is locally finite by
  Corollary~\ref{co:quot}, and the inverse map $\bfT(\U/\V)\to[\V,\U]$
  takes $\X$ to its inverse image under the localisation functor
  $\U\to\U/\V$.
\end{proof}

This proposition says that the lattice is \emph{relatively
  complemented}, that is, each interval is complemented. A lattice $L$
is \emph{complemented} if for each $x\in L$ there exists $y\in L$ such
that $x\vee y=1$ and $x\wedge y=0$.

If $\T$ is locally finite and admits a Nakayama functor
$N\colon\T\xto{\sim}\T$, then this induces a lattice automorphism of
$\bfT(\T)$ by taking a thick subcategory $\U$ to $N\U$. It follows
from the definition of $N$ that $^\perp N\U=\U^\perp$. Thus
$N\U=(\U^\perp)^\perp$ by Corollary~\ref{co:perp}.

\subsection*{Finiteness}
There are further finiteness results for the lattice $\bfT(\T)$ if
$\T$ is locally finite.

\begin{prop}\label{pr:lattice}
Let $\T$ be an essentially small triangulated category. If $\T$ is
finitely generated and locally finite, then $\T$ has only finitely
many thick subcategories.
\end{prop}
\begin{proof}
Let $\T=\Thick(X)$. For each thick subcategory $\U$ let $X_\U\to X$ be
the universal morphism from an object in $\U$ ending at $X$, which
exists by Theorem~\ref{th:adj}. Note that $\U=\Thick(X_\U)$ since
$\T=\Thick(X)$. Each indecomposable direct summand $X'$ of $X_\U$
satisfies $\Hom_\T(X',X)\neq 0$. There are only finitely many
isomorphism classes of such objects by Proposition~\ref{pr:fin}. It
follows that $\bfT(\T)$ is finite.
\end{proof}

Auslander--Reiten theory is used in an essential way for proving the
following.

\begin{thm}\label{th:fingen}
Let $\T$ be an essentially small and locally finite triangulated
category.  If $\T$ is connected then $\T$ is finitely generated and
has therefore only finitely many thick subcategories.
\end{thm}

We need the following lemma.
 
\begin{lem}[Amiot]\label{le:tree}
Let $\T$ be an essentially small triangulated category that is locally
finite. Then any connected component of the Auslander--Reiten quiver
of $\T$ is -- after removing possible loops -- of the form $\bbZ
\De/G$ for some finite tree $\De$ and some group $G$ of automorphisms
of $\bbZ\De$.
\end{lem}
\begin{proof}
If $X\to X$ is an irreducible morphism in $\T$ then $X=\t X$, by
\cite[Proposition~4.1.1]{Am2007}.  After removing such arrows $X\to
X$, the Auslander--Reiten quiver of $\T$ is a stable translation
quiver and therefore of the form $\bbZ \De/G$ by Riedtmann's
Struktursatz \cite{Ri1980}. The tree is finite by
\cite[Lemma~4.2.2]{Am2007}.
\end{proof}

\begin{proof}[Proof of Theorem~\ref{th:fingen}]
The Auslander--Reiten quiver of $\T$ is connected by
Proposition~\ref{pr:connect} and therefore -- modulo loops -- of the
form $\bbZ \De/G$ for some finite tree $\De$ and some group $G$ of
automorphisms by Lemma~\ref{le:tree}.  Note that $\bbZ\De$ does not
depend on the orientation of $\De$. Thus we may assume that $\De$ has
no path of length $>1$.  Then there are two types of vertices. Call a
vertex $x\in\De_0$ \emph{minimal} if there is no arrow from $\De$
ending at $x$, and \emph{maximal} if there is no arrow from $\De$
starting at $x$.

Let $\pi\colon\bbZ\De\to\bbZ\De/G$ be the canonical map and write each
vertex of $\bbZ\De$ as $\t^r x$ with $r\in\bbZ$ and $x\in\De_0$.  We
claim that $\T=\Thick(T)$ for $T=\coprod_{x\in\De_0}\pi(x)$. Thus we
fix an indecomposable object $X=\pi(y)$ in $\T$ and need to show that
$X$ belongs to $\Thick(T)$. Define for $y=\t^r x$
\[d(y)=\begin{cases} 2r,\quad&\text{if $x$ is maximal},\\
2r-1,&\text{if $x$ is minimal},
\end{cases}\]
and use induction on $d=d(y)$ as follows.  The cases $d=0$ and $d=-1$
are clear. Suppose now $d>0$ and consider the Auslander--Reiten
triangle $X\to Y\to \t^{-1} X\to \Si X$ starting at $X$. Then each
indecomposable direct summand of $Y$ is of the form $\pi(y')$ with an
arrow $y\to y'$ in $\bbZ\De$, by Lemma~\ref{le:AR}. We have
$d(y')=d-1$ and $d(\t^{-1}y)=d-2$. Thus $Y$ and $\t^{-1}X$ belong to
$\Thick(T)$, and it follows that $X$ is in $\Thick(T)$. The case $d<0$
is similar.

The finiteness of $\bfT(\T)$ follows from Proposition~\ref{pr:lattice}.
\end{proof}

\section{Simply connected triangulated categories}

Let $\T$ be a triangulated category that is essentially small and
locally finite.  We say that $\T$ is \emph{simply connected} if
the Auslander--Reiten quiver of $\T$ is connected and has no oriented
cycle. Here, an \emph{oriented cycle} is a path of length $>0$
starting and ending at the same vertex.  Note that there exists an
oriented cycle if and only if there is a chain of non-invertible
non-zero morphisms $X_0\to X_1\to\cdots\to X_n=X_0$ between
indecomposable objects; see Proposition~\ref{pr:connect}.

\subsection*{Rings of finite representation type}

For any ring $A$, we denote by $\mod A$ the category of finitely
presented $A$-modules. Recall that $A$ has \emph{finite representation
  type} if $A$ is artinian and there are only finitely many
isomorphism classes of finitely presented indecomposable $A$-modules.

\begin{thm}\label{th:frt}
For a triangulated category $\T$ the following are equivalent.
\begin{enumerate}
\item The triangulated category $\T$ is essentially small, algebraic,
  locally finite, and simply connected.
\item There exists a connected hereditary artinian ring $A$ of finite
  representation type such that $\T$ is equivalent to the bounded
  derived category $\bfD^b(\mod A)$.
\end{enumerate}
In this case the Auslander--Reiten quiver of $\T$ is of the form
$\bbZ\De$ for some tree $\De$ of Dynkin type.
\end{thm}

Let us explain how the ring $A$ is obtained from $\T$.  A
\emph{species} $(K_i,{_iE_j})_{i,j\in I}$ consists of a family of
division rings $K_i$ and a family of $K_i-K_j$-bimodules $_iE_j$.
There is an associated tensor algebra $\bigoplus_{n\ge 0}M^n$,
where \[M^0=\prod_{i\in I}K_i,\quad M^1=\bigoplus_{i,j\in I}{_iE_j},
\quad M^n=M^1\otimes_{M^0}\ldots\otimes_{M^0}M^1\quad\text{for}\quad
n>1,\] and multiplication is induced by the tensor product.

Let $\T$ be a triangulated category and suppose that $\T$ is locally
finite and simply connected.  Then the Auslander--Reiten quiver is of
the form $\bbZ\De$ for some finite tree by Lemma~\ref{le:tree}. In
fact, we may assume that $\De$ has no path of length $>1$, since
$\bbZ\De$ does not depend on the orientation of $\De$.  Choose
indecomposable objects $T_i$ for each vertex $i\in\De_0$. Define
$K_i=\End_\T(T_i)$ and $_iE_j=\Hom_\T(T_j,T_i)$ for
$i,j\in\De_0$. Observe that each $K_i$ is a division ring by
Proposition~\ref{pr:connect}, and that each bimodule $_iE_j$ is of
finite length on either side by Proposition~\ref{pr:fin}.  Then the
tensor algebra corresponding to the species $(K_i,{_iE_j})_{i,j\in
  \De_0}$ is isomorphic to the endomorphism ring of
$T=\coprod_{i\in\De_0}T_i$. The proof is straightforward, using the
fact that there are no paths of length $>1$ in $\De$.  We denote this
ring by $A$. Observe that $A$ is hereditary and artinian, since
$K=\prod_{i\in \De_0}K_i$ is semisimple and $\bigoplus_{i,j\in\De_0}{_iE_j}$
is finitely generated over $K$ on either side.

\begin{lem} 
The object $T$ is a tilting object, that is, $\T$ is generated by $T$ and 
$\Hom_\T(T,\Si^n T)=0$ for all $n\neq 0$. 
\end{lem}
\begin{proof}
We use Proposition~\ref{pr:connect} which gives a path in the
Auslander--Reiten quiver for each non-zero morphism between
indecomposable objects. In particular, there is for each
indecomposable object $X$ a path $X\to Y\to \t^{-1} X\to \cdots\to\Si X$, provided there
is a non-zero morphism starting in $X$ which is not a section.

Suppose that $\Hom_\T(T_i,\Si^n T_j)\neq 0$ and consider the following two cases:

$n>0$. From Remark~\ref{re:tau} one has $\Hom_\T(\Si^{n-1}
T_j,\t T_i)\neq 0$. Thus there is in $\bbZ\De$ a
path \[T_j\to\cdots\to\Si^{n-1}T_j\to\cdots\to \t T_i\to T_i'\to T_i.\]

$n<0$. There is in $\bbZ\De$ a path \[T_i\to\cdots\to\Si^n T_j\to\cdots\to \t
T_j\to T_j'\to T_j.\] 

In both cases, this contradicts the choice of $T_i,T_j\in\De_0$ and the
fact that $\De$ has no path of length $>1$. It follows that $n=0$.

The fact that $\T=\Thick(T)$ follows from the proof of
Theorem~\ref{th:fingen}.
\end{proof}

\begin{proof}[Proof of Theorem~\ref{th:frt}]
(1) $\Rightarrow$ (2): The assumptions on $\T$ yield a tilting object
  $T=\coprod_{i\in\De_0}T_i$, and its endomorphism ring $A=\End_\T(T)$
  is hereditary artinian.  It follows that there are equivalences
\[ \T\xleftarrow{\sim}\bfK^b(\add T)\xto{\sim}\bfD^b(\mod A).\]
For the first equivalence, see \cite[2.1]{Ke1990}, while the second is
clear from the fact that $\Hom_\T(T,-)$ identifies $\add T$ with the
category of finitely generated projective $A$-modules.  Locally
finiteness of $\T$ implies that $A$ is of finite representation type.

(2) $\Rightarrow$ (1): The abelian category $\mod A$ is hereditary,
that is, $\Ext_A^p(-,-)$ vanishes for $p>1$. It follows that each
indecomposable complex is concentrated in a single degree.  Finite
representation type of $A$ implies that $\T$ is locally finite, by
Proposition~\ref{pr:fin}. The Auslander--Reiten triangles in
$\bfD^b(\mod A)$ are obtained from almost split seqences in $\mod
A$. Thus the Auslander--Reiten quiver is of the form $\bbZ\De$ for
some tree $\De$; see \cite[\S4]{Ha1987} for details. In particular,
the quiver has no oriented cycles.

For the shape of $\De$, see \cite[\S4]{DRS}, where hereditary rings of
finite representation type are discussed.
\end{proof}

\section{Thick subcategories and non-crossing partitions}

Let $A$ be a \emph{hereditary Artin algebra}. Thus $A$ is an artinian
ring that is finitely generated over its centre and $\Ext^p_A(-,-)=0$
for all $p>1$.  Note that the centre of a hereditary Artin algebra is
semisimple. The algebra $A$ is said to be \emph{connected} if the
centre of $A$ is a field. From now on assume that $A$ is connected and
we denote the centre by $k$.

In this section, we associate to $A$ a poset of non-crossing
partitions and establish a correspondence between this poset and the
lattice of thick subcatgories of $\bfD^b(\mod A)$. This correspondence
was first observed for path algebras of finite and affine type by
Ingalls and Thomas \cite{IT}. For an introduction to non-crossing
partitions, see \cite{Ar2009}.

\subsection*{The Weyl group}
Fix a representative set $S_1,\ldots,S_n$ of simple $A$-modules.  We
associate to $A$ a generalised Cartan matrix $C(A)=(C_{ij})_{1\le
  i,j\le n}$ as follows.  Given two simple modules $S_i,S_j$, we have
$\Ext_A^1(S_i,S_j)=0$ or $\Ext_A^1(S_j,S_i)=0$. Assume $i\neq j$ and
$\Ext_A^1(S_j,S_i)=0$.  Then define
\[C_{ij}=-\ell_{\End_A(S_i)}(\Ext_A^1(S_i,S_j))\quad\text{and}\quad
C_{ji}=-\ell_{\End_A(S_j)}(\Ext_A^1(S_i,S_j)).\] In addition, define
$C_{ii}=2$ and $d_i=\ell_k(\End_A(S_i))$ for each $i$. Then one has
$d_i C_{ij}=d_jC_{ji}$. Thus $C(A)$ is a symmetrisable generalised
Cartan matrix in the sense of \cite{Ka1985}.

Next we consider the Weyl group corresponding to $C(A)$. Let $\bbR^n$
denote the $n$-dimensional real space with standard basis
$\e_1,\ldots,\e_n$. Define a symmetric bilinear form
by $(\e_i,\e_j)=d_iC_{ij}$ and for each $\a\in\bbR^n$ with
$(\a,\a)\neq 0$ the reflection 
\[s_\a\colon\bbR^n\lto\bbR^n,\quad \xi\mapsto\xi-2\frac{(\xi,\a)}{(\a,\a)}\a.\]
We write $s_i$ for the \emph{simple reflection} $s_{\e_i}$ and observe
that each $s_i$ maps $\bbZ^n$ into itself.  The \emph{Weyl group} is
the group $W$ generated by the simple reflections $s_1,\ldots, s_n$.
The \emph{real roots} are by definition the elements of $\bbZ^n$ of
the form $w(\e_i)$ for some $w\in W$ and some $i\in\{1,\ldots,n\}$.
Note that for any real root $\a$ the corresponding reflection $s_\a$
belongs to $W$ since $s_{w(\a)}=ws_\a w^{-1}$.

We define the \emph{absolute order} on $W$ with respect to the
\emph{absolute length} $\ell$ as follows. Consider the set of
reflections
\[W_1=\{ws_iw^{-1}\mid w\in W,\,1\le i\le n\}\] and 
for each $w\in W$ let $\ell(w)$ denote the minimal $r\ge 0$ such that
$w$ can be written as product $w=x_1\ldots x_r$ of reflections $x_j\in
W_1$. Given $u,v\in W$ define
\[u\le v \quad\iff\quad \ell (u)+\ell(u^{-1}v)=\ell(v).\]
Note that the length function $\ell$ and the absolute order are
invariant under conjugation with a fixed element of $W$.

A \emph{Coxeter element} in $W$ is any element of $W$ that is
conjugate to one of the form $s_{\s (1)}s_{\s (2)}\ldots s_{\s (n)}$
for some permutation $\s$. Note that $\ell(c)=n$ for each Coxeter
element $c$ by \cite[Theorem~1.1]{Dy2001}. This has the following
immediate consequence.

\begin{lem}\label{le:cox}
  \pushQED{\qed} Let $c$ be a Coxeter element and $x_1,\ldots,x_n$
  a sequence of reflections in $W_1$ such that $c=x_1\cdots
  x_n$. If $1\le r\le s\le n$, then \[\ell(x_1\ldots
  x_r)=r\quad\text{and}\quad x_1\ldots x_r\le x_1\ldots x_s.
  \qedhere\]
\end{lem}

Relative to a Coxeter element $c$ one defines the poset of
\emph{non-crossing partitions}
\[\NC(W,c)=\{w\in W\mid \id \le w\le c\}.\]
Given two Coxeter elements $c,c'$ in $W$, we have
$\NC(W,c)\cong\NC(W,c')$ provided that $c$ and $c'$ are conjugate.

The Grothendieck group $K_0(\mod A)$ is isomorphic to $\bbZ^n$ via the
map sending each simple $A$-module $S_i$ to the standard base vector
$\e_i$. The image of an $A$-module $X$ under this map is called
\emph{dimension vector} and denoted by $\dim X$; we write $s_X=s_{\dim
  X}$ for the corresponding reflection.

\subsection*{Exceptional modules and sequences}

An $A$-module $X$ is called \emph{exceptional} if $X$ is
indecomposable and $\Ext_A^1(X,X)=0$. Note that $\dim
X$ is a real root if $X$ is exceptional \cite[Corollary~2]{Ri1994}.  A
sequence $(X_1,\ldots,X_r)$ of $A$-modules is called
\emph{exceptional} if each $X_i$ is exceptional and
\[\Hom_A(X_j,X_i)=0=\Ext_A^1(X_j,X_i)\quad\text{for all}\quad i<j.\]
Such a sequence is \emph{complete} if $r=n$.

\begin{lem}\label{le:inj}
Sending an $A$-module $X$ to the reflection $s_X =s_{\dim X}$ gives
an injective map from the set of isomorphism classes of exceptional
$A$-modules into $W$.
\end{lem}
\begin{proof}
An exceptional $A$-module is uniquely determined by its dimension
vector; see \cite[Lemma~8.2]{Ke1996}. On the other hand, given a
reflection $s\in W_1$, we have $s=s_\a$ where $\a\in\bbZ^n$ is the
unique vector with non-negative entries satisfying $s(\a)=-\a$.
Thus $s_{\dim X}=s_{\dim Y}$ implies $X\cong Y$.
\end{proof}

\begin{thm}[Crawley-Boevey, Ringel, Igusa--Schiffler]\label{th:braid}
Let $A$ be a connected hereditary Artin algebra with simple modules
$S_1,\ldots,S_n$ satisfying $\Ext^1_A(S_j,S_i)=0$ for all $i<j$. Denote by
$W$ the associated Weyl group and fix the Coxeter element
$c=s_{1}\cdots s_{n}$. Then the braid group $B_n$ on $n$
strings acts transitively on
\begin{enumerate}
\item[---] the isomorphism classes of complete exceptional
  sequences $(X_1,\ldots,X_n)$ in $\mod A$, and
\item[---] the sequences $(x_1,\ldots,x_n)$ of reflections in $W_1$
  such that $c=x_1\cdots x_n$.
\end{enumerate}
Moreover, $\s (X_1,\ldots,X_n)=(Y_1,\ldots,Y_n)$ implies
 $\s (s_{X_1},\ldots,s_{X_n})=(s_{Y_1},\ldots,s_{Y_n})$ for all $\s\in B_n$.
\end{thm}
\begin{proof}
For the action of the braid group on complete exceptional sequences,
see \cite{CB1992, Ri1994}. For the action on factorisations of the
Coxeter element, see \cite[Theorem~1.4]{IS2010}. The compatibility of
both actions follows from the computation of the dimension vectors of
the modules in an exceptional sequence under the braid group action;
see the proof of \cite[Corollary~2.4]{IS2010} and also the Corollary
in \cite{CB1992}.
\end{proof}

\begin{cor}\label{co:braid}
  Let $(x_1,\ldots,x_n)$ be a sequence of reflections in $W_1$ such
  that $c=x_1\cdots x_n$. Then there exists up to
  isomorphism a unique complete exceptional sequence
  $(X_1,\ldots,X_n)$ such that $x_i=s_{X_i}$ for all $i$.
\end{cor}
\begin{proof}
Theorem~\ref{th:braid} gives $\s\in B_n$ such that
$(x_1,\ldots,x_n)=\s(s_1,\ldots,s_n)$. Let $(X_1,\ldots,
X_n)=\s(S_1,\ldots,S_n)$. Then $x_i=s_{X_i}$ for all $i$. Uniqueness
follows from Lemma~\ref{le:inj}.
\end{proof}

\subsection*{Thick subcategories}
We consider the bounded derived category $\bfD^b(\mod A)$ and identify
$A$-modules with complexes concentrated in degree zero. Recall the
following correspondence.

\begin{prop}[Br\"uning]\label{pr:exact}
Let $\A$ be a hereditary abelian category. The canonical inclusion
$\A\to\bfD^b(\A)$ induces a bijection between
\begin{enumerate}
\item[--] the set of exact abelian and extension closed subcategories of $\A$, and
\item[--] the set of thick subcategories of $\bfD^b(\A)$.
\end{enumerate}
The bijection sends $\C\subseteq\A$ to $\{X\in\bfD^b(\A)\mid
H^iX\in\C\text{ for all }i\in\bbZ\}$.
Its inverse sends $\U\subseteq\bfD^b(\A)$ to 
\[H^0\U=\{Y\in\A\mid Y\cong H^0X\text{ for some }X\in\U\}.\]
\end{prop}
\begin{proof}
See \cite[Theorem~5.1]{Br}.
\end{proof}

The \emph{Loewy length} of an object $X$ in some abelian category is
the smallest $p\ge 0$ such that there exists a chain $0=X_0\subseteq
X_1\subseteq\ldots \subseteq X_p=X$ so that $X_{i+1}/X_i$ is semisimple
for all $i$. The \emph{height} of an abelian category is the supremum
of the Loewy lengths of its objects.

Next we characterise the thick subcategories of $\bfD^b(\mod A)$ such
that the inclusion admits an adjoint. The connection with exceptional
sequences is due to Bondal \cite{Bo1990}.

\begin{prop}\label{pr:hered}
Let $A$ be a hereditary Artin algebra. For a thick subcategory $\U$ of
$\bfD^b(\mod A)$ are equivalent.
\begin{enumerate}
\item The inclusion $\U\to \bfD^b(\mod A)$ admits a left adjoint.
\item The inclusion $H^0\U\to \mod A$ admits a left adjoint.
\item The abelian category $H^0\U$ is of finite height and has only
  finitely many isomorphism classes of simple objects.
\item There exists an exceptional sequence $(X_1,\ldots,X_r)$ in $\mod
  A$  such that $\U=\Thick(X_{1},\ldots,X_r)$.
\item There exists a complete exceptional sequence $(X_1,\ldots,X_n)$
  in $\mod A$ such that
\[\U=\Thick(X_{1},\ldots,X_r)\quad\text{and}\quad^\perp\U=\Thick(X_{r+1},\ldots,X_n)\]
\end{enumerate}
\end{prop}

\begin{rem} The $k$-duality $(\mod A)^\op\xto{\sim}\mod(A^\op)$
  induces a duality
\[\bfD^b(\mod
A)^\op\stackrel{\sim}\lto\bfD^b(\mod A^\op)\] which preserves the
property (3). Thus the existence of left adjoints in (1) and (2) is
equivalent to the existence of right adjoints. 
\end{rem}

\begin{rem}
  The equivalent conditions in Proposition~\ref{pr:hered} are
  automatically satisfied if the algebra $A$ is of finite
  representation type; see Theorem~\ref{th:adj}.
\end{rem}

\begin{proof}[Proof of Proposition~\ref{pr:hered}]
(1) $\Leftrightarrow$ (2): See \cite[\S2]{KS2010}.

(2) $\Rightarrow$ (3): A left adjoint $F\colon\mod A\to H^0\U$ sends a
  projective generator to a projective generator. Thus one takes
  $B=\End_A(FA)$ and gets an equivalence 
\[\Hom_A(FA,-)\colon H^0\U\stackrel{\sim}\lto\mod B.\]
Clearly, $\mod B$ has finite height and only finitely many simple objects.

(3) $\Rightarrow$ (4): It follows from \cite[8.2]{Ga1973} that the
category $H^0\U$ is equivalent to $\mod B$ for some finite dimensional
$k$-algebra $B$. The algebra $B$ is hereditary since $A$ is
hereditary. Thus the simple $B$-modules form a complete exceptional
sequence $(X_1,\ldots,X_r)$ in $\mod B$. This gives an exceptional
sequence in $\mod A$ satisfying $\U=\Thick(X_{1},\ldots,X_r)$.

(4) $\Rightarrow$ (1): See  \cite[Theorem~3.2]{Bo1990}.

(1)--(4) $\Rightarrow$ (5): The duality provides a right adjoint for
the inclusion of $\U$.  Thus the inclusion of $^\perp\U$ admits a left
adjoint by Lemma~\ref{le:tria-loc}.  It follows that there is an
exceptional sequence $(X_{r+1},\ldots,X_s)$ in $\mod A$ with
$^\perp\U=\Thick(X_{r+1},\ldots,X_s)$.  Then $(X_{1},\ldots,X_s)$ is
an exceptional sequence satisfying
$\Thick(X_{1},\ldots,X_s)=\bfD^b(\mod A)$ which is therefore complete.

(5) $\Rightarrow$ (4): Clear.
\end{proof}

\begin{exm}
For a tame hereditary algebra $A$, the regular $A$-modules form an exact abelian
and extension closed subcategory of $\mod A$ that is of infinite height.
\end{exm}

\subsection*{A classification}
The following result is a consequence of Theorem~\ref{th:braid} and
provides a combinatorial classification of the thick subcategories of
$\bfD^b(\mod A)$ satisfying the equivalent conditions in
Proposition~\ref{pr:hered}. Using the correspondence in
Proposition~\ref{pr:exact}, this translates into a classification of
certain abelian subcategories of $\mod A$.

For the path algebra of a quiver over an algebraically closed field,
this result is due to Igusa, Schiffler, and Thomas \cite{IS2010},
Previous work of Ingalls and Thomas \cite{IT} establishes the result
in finite and affine type.

\begin{thm}\label{th:classif}
Let $A$ be a connected hereditary Artin algebra with simple modules
$S_1,\ldots,S_n$ satisfying $\Ext^1_A(S_j,S_i)=0$ for all
$i<j$. Denote by $W$ the associated Weyl group and fix the Coxeter
element $c=s_{1}\cdots s_{n}$. Then there exists an order
preserving bijection between
\begin{enumerate}
\item[---] the set of thick subcategories of $\bfD^b(\mod A)$ such
  that the inclusion admits a left adjoint, and
\item[---] the set of non-crossing partitions $\NC(W,c)$.
\end{enumerate}
The map sends a thick subcategory which is
generated by an exceptional sequence $(X_1,\ldots,X_r)$ to
$s_{X_1}\cdots s_{X_r}$.
\end{thm}

Let us formulate an immediate consequence.

\begin{cor}
 \pushQED{\qed} Let $(X_1,\ldots, X_r)$ and $(Y_1,\ldots, Y_s)$ be
 exceptional sequences in $\mod A$. Then
\[\Thick(X_1,\ldots,X_r)=\Thick(Y_1,\ldots, Y_s)\quad\Longleftrightarrow\quad
s_{X_1}\cdots s_{X_r}=s_{Y_1}\cdots s_{Y_s}.\qedhere\]
\end{cor}

\begin{proof}[Proof of Theorem~\ref{th:classif}]
Fix a thick subcategory $\U\subseteq\bfD^b(\mod A)$
such that the inclusion admits a left adjoint. There exists a complete
exceptional sequence $(X_1,\ldots,X_n)$ in $\mod A$ such that
$\U=\Thick(X_{1},\ldots,X_r)$ for some $r\le n$, by
Proposition~\ref{pr:hered}. We assign to $\U$ the element
$\cox(\U)=s_{X_{1}}\cdots s_{X_r}$ in $W$. Observe that $\cox(\U)\le
c$ by Lemma~\ref{le:cox}, since $s_{X_{1}}\cdots s_{X_n}=c$ by
Theorem~\ref{th:braid}. Thus $\cox$ gives a map into $\NC(W,c)$.

\emph{The map $\cox$ is well-defined:} Choose a second exceptional
sequence $(Y_1,\ldots,Y_s)$ in $\mod A$ such that
$\U=\Thick(Y_{1},\ldots,Y_s) $.  Then
$(Y_1,\ldots,Y_s,X_{r+1},\ldots,X_n)$ is a complete exceptional
sequence, and it follows from Theorem~\ref{th:braid} that
\[s_{Y_1}\cdots s_{Y_s}s_{X_{r+1}}\cdots s_{X_n}=c=s_{X_1}\cdots s_{X_r}s_{X_{r+1}}\cdots s_{X_n}.\]
Thus $s_{Y_{1}}\cdots s_{Y_s}=s_{X_{1}}\cdots s_{X_r}$.

\emph{The map $\cox$ is injective:} Let $\U$ and $\V$ be thick
subcategories such that $\cox (\U)=\cox(\V)$. Thus there are two
complete exceptional sequences $(X_1,\ldots,X_n)$ and
$(Y_1,\ldots,Y_n)$ such that $\U=\Thick(X_1,\ldots,X_r)$ and
$\V=\Thick(Y_1,\ldots,Y_s)$ for some pair of integers $r,s\le
n$. Moreover, $s_{X_1}\cdots s_{X_r}=s_{Y_1}\cdots s_{Y_s}$.  It
follows that \[r=\ell(s_{X_1}\cdots s_{X_r})=\ell(s_{Y_1}\cdots
s_{Y_s})=s\] and therefore $c=s_{Y_1}\cdots s_{Y_s}s_{X_{r+1}}\cdots
s_{X_n}$.  From Corollary~\ref{co:braid} and Lemma~\ref{le:inj} one
gets that $(Y_1,\ldots,Y_s,X_{r+1},\ldots, X_n)$ is a complete
exceptional sequence. Thus
\[\U=\Thick(X_1,\ldots,X_r)=\Thick(X_{r+1},\ldots,X_n)^\perp=\Thick(Y_1,\ldots,Y_s)=\V.\]

\emph{The map $\cox$ is surjective:} Let $w\in NC(W,c)$. Thus we can write
\[c=x_1\cdots x_rx_{r+1}\cdots x_n\] as products of reflections $x_i\in
W_1$ such that $w=x_1\cdots x_r$. From Corollary~\ref{co:braid}
one gets a complete exceptional sequence $(X_1,\ldots, X_n)$ with
$x_i=s_{X_i}$ for all $i$. Set $\U=\Thick(X_1,\ldots,X_r)$. Then
$\cox(\U)=w$.

\emph{The map $\cox$ is order preserving:} Let $\V\subseteq\U$ be
thick subcategories. From Proposition~\ref{pr:hered} one gets a
complete exceptional sequence $(X_1,\ldots,X_n)$ such that
\[\V=\Thick(X_{1},\ldots,X_s)\quad\text{and}\quad\U=\Thick(X_{1},\ldots,X_r)\] for
some $s\le r\le n$.  It follows from Lemma~\ref{le:cox} that
\[\cox(\V)=s_{X_{1}}\cdots s_{X_s}\le s_{X_{1}}\cdots
s_{X_r}=\cox(\U).\qedhere\]
\end{proof}

\begin{rem}
Let $\U\subseteq\bfD^b(\mod A)$ be a thick subcategory such that the inclusion admits a left adjoint.
Then $\cox (\U)\cox({^\perp\U})=c$.
\end{rem}

\subsection*{Example: The Kronecker algebra}

Let $k$ be a field and consider the \emph{Kronecker algebra}, that is,
the path algebra of the quiver
$\xymatrix@=15pt{\scriptscriptstyle{\circ}
  \ar@<.5ex>[r]\ar@<-.5ex>[r]&{\scriptscriptstyle{\circ}}}$. This is a
tame hereditary Artin algebra;  we denote it by $K$ and
compute the lattice of thick subcategories of $\bfD^b(\mod K)$. For
a description of $\mod K$, we refer to \cite[VIII.7]{ARS}.

Each finite dimensional indecomposable $K$-module is either
exceptional or regular. The dimension vectors of the exceptional
$K$-modules are $(p,q)\in\bbZ^2$ with $p,q\ge 0$ and $|p-q|=1$. Thus
the non-crossing partitions with respect to the Coxeter element
$c=s_1s_2$ form the lattice
\[\NC(W,c)=\{s_{(p,q)}\mid p,q\ge 0\text{ and }|p-q|=1 \}\cup \{\id,c\}\]
with the following Hasse diagram:
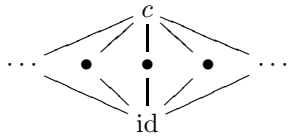
\[\xymatrix@=1em{
&&c\ar@{-}[dll]\ar@{-}[dl]\ar@{-}[dr]\ar@{-}[drr]\ar@{-}[d]\\
\cdots\ar@{-}[drr]&\bullet\ar@{-}[dr]&\bullet\ar@{-}[d]&\bullet\ar@{-}[dl]&\cdots\ar@{-}[dll]\\
&&\id
}\]

The regular $K$-modules form an extension closed exact abelian
subcategory of $\mod K$ that is uniserial. The simple objects of this
abelian category are parameterised by the closed points of the
projective line over $k$, which we identify with non-maximal and
non-zero homogeneous prime ideals $\frp\subseteq k[x,y]$.  We denote
the set of closed points by $\bbP^1(k)$ and write $\mathbf
2^{\bbP^1(k)}$ for its power set. Adding an extra greatest element
(the set of all non-maximal homogeneous prime ideals) to
$\mathbf 2^{\bbP^1(k)}$ yields a new lattice which we denote by
$\widehat{\mathbf 2}^{\bbP^1(k)}$. The simple object corresponding to
$\frp$ is denoted by $S_\frp$. We obtain an injective map
\[\widehat{\mathbf 2}^{\bbP^1(k)}\lto\bfT(\bfD^b(\mod K))\]
by sending $U\subseteq \bbP^1(k)$ to $\Thick(\{S_\frp\mid\frp\in U\})$
and $\bbP^1(k)\cup \{0\}$ to $\bfD^b(\mod K)$.

Let $L',L''$ be a pair of lattices with smallest elements $0',0''$ and
greatest elements $1',1''$. Denote by $L'\amalg L''$ the new lattice
which is obtained from the disjoint union $L'\cup L''$ (viewed as sum
of posets) by identifying $0'=0''$ and $1'=1''$.

\begin{prop}
The lattice of thick subcategories of $\bfD^b(\mod K)$ is isomorphic to
the lattice $\NC(W,c)\amalg \widehat{\mathbf2}^{\bbP^1(k)}$.
\end{prop}
\begin{proof}
We write $\bfT(K)= \bfT(\bfD^b(\mod K))$ and have injective maps
\[\NC(W,c)\lto \bfT(K)\quad\text{and}\quad
\widehat{\mathbf2}^{\bbP^1(k)}\lto\bfT(K)\] which induce an order preserving
map \[\NC(W,c)\amalg \widehat{\mathbf2}^{\bbP^1(k)}\lto\bfT(K).\] In
order to prove bijectivity, fix a pair of indecomposable $K$-modules
$X,Y$ such that $X$ is exceptional and $Y$ is regular.  We have
$\Thick(X)\cap\Thick(Y)=0$ and this gives injectivity; surjectivity
follows from the fact that $\Thick(X,Y)=\bfD^b(\mod K)$.
\end{proof}

The category of coherent sheaves on the projective line $\bbP^1_k$
admits a tilting object $T=\Oc(0)\oplus\Oc(1)$ such that $\End
(T)\cong K$; see \cite{Be1978}. Thus $\RHom(T,-)$ induces a triangle
equivalence
\[\bfD^b(\coh\bbP^1_k)\stackrel{\sim}\lto\bfD^b(\mod K)\]
and this yields a description of the lattice of thick subcategories of
$\bfD^b(\coh\bbP^1_k)$. Note that the category of coherent sheaves
carries a tensor product. The thick tensor subcategories have been
classified by Thomason \cite[Theorem~3.5]{Th1997}; they are precisely
the ones parameterised by subsets of $\bbP^1(k)$.

\begin{appendix}
\section{Auslander--Reiten theory}\label{se:appendix}

In this appendix we collect some basic facts from Auslander--Reiten
theory, as initiated by Auslander and Reiten in
\cite{AR1975}. Krull--Remak--Schmidt categories form the appropriate
setting for this theory, while exact or triangulated structures are
irrelevant for most parts; see also \cite{Ba,Li}. This material is
well-known, at least for categories that are linear over a field with
finite dimensional morphism spaces.  We provide full proofs for most
statements, including references whenever they are available.

Let $\C$ be an essentially small additive category. Then $\C$ is called
\emph{Krull--Remak--Schmidt category} if every object in $\C$ is a
finite coproduct of indecomposable objects with local endomorphism
rings.

It is convenient to view $\C$ as a ring with several objects. Thus we
use the category $\Mod\C$ of \emph{$\C$-modules}, which are by
definition the additive functors $\C^\op\to\Ab$ into the category of
abelian groups.

There is the following useful characterisation in terms of projective
covers.  Recall that a morphism $\p\colon P\to M$ is a
\emph{projective cover}, if $P$ is a projective object and $\p$ is an
\emph{essential epimorphism}, that is, a morphism $\a\colon P'\to P$
is an epimorphism if and only if $\p\a$ is an epimorphism.

\begin{prop}\label{pr:KRS}
For an essentially small additive category $\C$ with split idempotents the following are
equivalent.
\begin{enumerate}
\item The category $\C$ is a Krull--Remak--Schmidt category.
\item Every finitely generated $\C$-module admits a projective
  cover.
\end{enumerate} 
\end{prop}

The proof requires some preparations, and we begin with two lemmas.

\begin{lem}\label{le:procov}
Let $P$ be a projective object. Then the following are equivalent for
an epimorphism $\p\colon P\to M$.
\begin{enumerate}
\item The morphism $\p$ is a projective cover of $M$.
\item Every endomorphism $\a\colon P\to P$ satisfying $\p\a=\p$
is an isomorphism.
\end{enumerate}
\end{lem}
\begin{proof}
(1) $\Rightarrow$ (2): Let $\a\colon P\to P$ be an endomorphism
    satisfying $\p\a=\p$. Then $\a$ is an epimorphism since $\p$
    is essential. Thus there exists $\a'\colon P\to P$ satisfying
    $\a\a'=\id_P$ since $P$ is projective.  It follows that
    $\p\a'=\p$ and therefore $\a'$ is an epimorphism. On the
    other hand, $\a'$ is a monomorphism. Thus $\a'$ and $\a$ are
    isomorphisms.

(2) $\Rightarrow$ (1): Let $\a\colon P'\to P$ be a morphism such that
    $\p\a$ is an epimorphism. Then $\p$ factors through
    $\p\a$ via a morphism $\a'\colon P\to P'$ since $P$ is
    projective. The composite $\a\a'$ is an isomorphism and
    therefore $\a$ is an epimorphism. Thus $\p$ is essential.
\end{proof}

\begin{lem}
\label{le:covsimple}
Let $\p\colon P\to S$ be an epimorphism such that $P$ is projective
and $S$ is simple.  Then the following are equivalent.
\begin{enumerate}
\item The morphism $\p$ is a projective cover of $S$.
\item The object $P$ has a unique maximal subobject.
\item The endomorphism ring of $P$ is local.
\end{enumerate}
\end{lem}
\begin{proof}
(1) $\Rightarrow$ (2): Let $U\subseteq P$ be a subobject and suppose
    $U\not\subseteq \Ker\p$. Then $U+\Ker\p=P$, and therefore $U=P$
    since $\p$ is essential.  Thus $\Ker\p$ contains every proper
    subobject of $P$.

    (2) $\Rightarrow$ (3): First observe that $P$ is an indecomposable
    object.  It follows that every endomorphism of $P$ is invertible
    if and only if it is an epimorphism. Given two non-units $\a,\b$
    in $\End(P)$, we have therefore
    $\Im(\a+\b)\subseteq\Im\a+\Im\b\subsetneq P$.  Here we use that
    $P$ has a unique maximal subobject. Thus $\a+\b$ is a non-unit and
    $\End(P)$ is local.

(3) $\Rightarrow$ (1): Consider the $\End(P)$-submodule $H$ of
$\Hom(P,S)$ which is generated by $\p$. Suppose $\p=\p\a$ for some
$\a$ in $\End(P)$.  If $\a$ belongs to the Jacobson radical, then
$H=H J(\End(P))$, which is not possible by Nakayama's lemma. Thus
$\a$ is an isomorphism since $\End(P)$ is local. It follows from
Lemma~\ref{le:procov} that $\p$ is a projective cover.
\end{proof}

Given any $\C$-module $M$, we denote by $\rad M$ its \emph{radical},
that is, the intersection of all maximal submodules of $M$.

\begin{proof}[Proof of Proposition~\ref{pr:KRS}]
First observe that the Yoneda functor $\C\to\Mod\C$ taking an object
$X$ to $\Hom_\C(-,X)$ identifies $\C$ with the category of finitely
generated projective $\C$-modules. The other tools are
Lemmas~\ref{le:procov} and \ref{le:covsimple}, which are used without
further reference.

(1) $\Rightarrow$ (2): Let $M$ be a simple $\C$-module and choose an
indecomposable object $X$ with $M(X)\neq 0$. Then there exists a
non-zero morphism $\Hom_\C(-,X)\to M$ which is a projective cover
since $\End_\C(X)$ is local. Thus every finite  sum of simple
$\C$-modules admits a projective cover.

Now let $M$ be a finitely generated $\C$-module and choose an
epimorphism $\p\colon P\to M$ with $P$ finitely generated
projective. Let $P=\bigoplus_iP_i$ be a decomposition into
indecomposable modules. Then
\[P/\rad P=\bigoplus_iP_i/\rad P_i\] is a finite sum of simple
$\C$-modules since each $P_i$ has a local endomorphism ring. The
epimorphism $\p$ induces an epimorphism $P/\rad P\to M/\rad M$ and
therefore $M/\rad M$ decomposes into finitely many simple modules.
There exists a projective cover $Q\to M/\rad M$ and this factors
through the canonical morphism $\pi\colon M\to M/\rad M$ via a
morphism $\psi \colon Q\to M$. The morphism $\pi$ is essential since
$M$ is finitely generated. It follows that $\psi$ is an
epimorphism. The morphism $\psi$ is essential since $\pi\psi$ is
essential. Thus $\psi$ is a projective cover.

(2) $\Rightarrow$ (1): Let $X$ be an object in $\C$.  We show that $X$
admits a decomposition into finitely many indecomposable objects. Set
$P=\Hom_\C(-,X)$. We claim that $P/\rad P$ is semisimple. Choose a
quotient $P/U$, where $\rad P\subseteq U\subseteq P$, and let $Q\to
P/U$ be a projective cover. This gives maps $P\to Q$ and $Q\to P$. The
composite $Q\to P\to Q$ is invertible and induces an isomorphism
$P/U\to P/\rad P\to P/U$. Thus the canonical morphism $P/\rad P\to
P/U$ has a right inverse, and we conclude that $P/\rad P$ is
semisimple. Let $P/\rad P=\bigoplus_i S_i$ be a decomposition into
finitely many simple objects and choose a projective cover $P_i\to
S_i$ for each $i$. Then $P\cong \bigoplus_i P_i$, since $P\to P/\rad
P$ and $ \bigoplus_i P_i\to \bigoplus_i S_i$ are both projective
covers. It remains to observe that each $P_i$ is indecomposable with a
local endomorphism ring.
\end{proof}

\begin{rem}
Fix a ring $A$ and let $\C$ be the category of finitely generated
projective $A$-modules.  Then the functor $\Mod\C\to\Mod A$ taking $X$
to $X(A)$ is an equivalence. The category $\C$ is
Krull--Remak--Schmidt if and only if the ring $A$ is semiperfect.
\end{rem}

From now on suppose that $\C$ is a Krull--Remak--Schmidt category.

\subsection*{The radical}
The  \emph{radical} of $\C$ is by definition the collection of subgroups 
\[\Rad_\C(X,Y)\subseteq\Hom_\C(X,Y)\]
for each pair $X,Y$ of objects in $\C$, where
\begin{align*}
\Rad_\C(X,Y)
&=\{\p\in\Hom_\C(X,Y)\mid\id_X-\p'\p\text{ is invertible for all
}\p'\colon Y\to X\}\\
&=\{\p\in\Hom_\C(X,Y)\mid\id_Y-\p\p'\text{ is invertible for all
}\p'\colon Y\to X\}.
\end{align*} 
The radical is a \emph{two-sided ideal} of $\C$, that is, a subfunctor
of $\Hom_\C(-,-)\colon\C^\op\times\C\to\Ab$. It is actually the unique
two-sided ideal $\mathfrak J$ of $\C$ such that $\mathfrak J(X,X)$
equals the Jacobson radical of $\End_\C(X)$ for each object $X$; see
\cite{Ke1964}.

Given two decompositions $X=\coprod_i X_i$ and $Y=\coprod_j Y_j$ in $\C$, we
have \[\Rad_\C(X,Y)=\bigoplus_{i,j}\Rad_\C(X_i,Y_j).\] A morphism
between indecomposable objects belongs to the radical if and only if
it is not invertible, since indecomposable objects have local
endomorphism rings.

\subsection*{Simple functors}
Given an object $X$ in $\C$, the functor $\Rad_\C(-,X)$ equals the
intersection of all maximal subfunctors of
$\Hom_\C(-,X)\colon\C^\op\to\Ab$. This observation has the following
consequence, where
\[S_X=\Hom_\C(-,X)/\Rad_\C(-,X).\]

\begin{lem}[{\cite[Proposition~2.3]{Au1974}}]\label{le:simples}
The map sending an object $X$ of $\C$ to the functor $S_X$
induces (up to isomorphism) a bijection between the indecomposable
objects of $\C$ and the simple objects of $\Mod\C$.
\end{lem}
\begin{proof}
If an object $X$ has a local endomorphism ring, then $\Rad_\C(-,X)$ is
the unique maximal subobject of $\Hom_\C(-,X)$. Thus $S_X$ is simple
in that case, and the inverse map sends a simple object $S$ in
$\Mod\C$ to the unique indecomposable object $X$ in $\C$ such that
$S(X)\neq 0$.
\end{proof}

\subsection*{Irreducible morphisms}
The definition of the radical is extended recursively as follows. For
each $n> 1$ and each pair of objects $X,Y$ let $\Rad^n_\C(X,Y)$ be the
set of morphisms $\p\in\Hom_\C(X,Y)$ that admit a factorisation
$\p=\p''\p'$ with $\p'\in\Rad_\C(X,Z)$ and $\p''\in\Rad^{n-1}_\C(Z,Y)$
for some object $Z$. Then we set
\[\Irr_\C(X,Y)=\Rad_\C(X,Y)/\Rad^2_\C(X,Y).\] 
This is a bimodule over the rings $\De(X)$
and $\De(Y)$, where
\[\De(X)=\End_\C(X)/\Rad_\C(X,X).\]

Note that a morphism $X\to Y$ between indecomposable objects belongs
to $\Rad_\C(X,Y)\smallsetminus\Rad^2_\C(X,Y)$ if and only it is
irreducible. A morphism $\p$ is called \emph{irreducible} if $\p$ is
neither a section nor a retraction, and if $\p=\p''\p'$ is a
factorisation then $\p'$ is a section or $\p''$ is a retraction.

\begin{lem}\label{le:irr}
For each pair of indecomposable objects $X,Y$ in $\C$, we have
\[\Ext^1_\C(S_Y,S_X)\cong\Hom_{\De(X)}(\Irr_\C(X,Y),\De(X))\]
as bimodules over $\De(X)$ and $\De(Y)$.
\end{lem}
\begin{proof}
Applying $\Hom_\C(-,S_X)$ to the exact sequence
\[0\lto\Rad_\C(-,Y)\lto\Hom_\C(-,Y)\lto S_Y\lto 0\]
gives 
\begin{equation*}\label{eq:ext}
\Ext^1_\C(S_Y,S_X)\cong\Hom_\C(\Rad_\C(-,Y),S_X).
\end{equation*}
Then applying  $\Hom_\C(-,S_X)$ to the exact sequence
\[0\lto\Rad^2_\C(-,Y)\lto\Rad_\C(-,Y)\lto \Irr_\C(-,Y)\lto 0\]
gives \[\Ext^1_\C(S_Y,S_X)\cong\Hom_\C(\Irr_\C(-,Y),S_X)
\cong\Hom_{\De(X)}(\Irr_\C(X,Y),\De(X))\] since $S_X(X)=\De(X)$.
\end{proof}

\subsection*{Almost split morphisms}
A morphism $\p\colon X\to Y$ is called \emph{right almost split} if
$\p$ is not a retraction and if every morphism $X'\to Y$ that is not a
retraction factors through $\p$. The morphism $\p$ is \emph{right
  minimal} if every endomorphism $\a\colon X\to X$ with $\p\a=\p$ is
invertible. Note that $Y$ is indecomposable if $\p$ is right almost
split.  \emph{Left almost split} morphisms and \emph{left minimal}
morphisms are defined dually. The term \emph{minimal right almost
  split} means \emph{right minimal} and \emph{right almost split}.

Recall that a projective presentation
\[P_n\stackrel{\d_n}\lto P_{n-1}\stackrel{\d_{n-1}}\lto\cdots \stackrel{\d_{2}}\lto P_1
\stackrel{\d_{1}}\lto P_0\stackrel{\d_0}\lto M\lto 0\]
is \emph{minimal} if each morphism $P_i\to\Im\d_i$ is a projective cover.

\begin{lem}[{\cite[Chap.~II, Proposition~2.7]{Au1978}}]\label{le:cover}
A  morphism $X\to Y$ in $\C$ is minimal right almost split
if and only if it induces in $\Mod\C$ a minimal projective
presentation
\[\Hom_\C(-,X)\lto\Hom_\C(-,Y)\lto S\lto 0\]
of a simple object.
\end{lem}

\begin{proof}
A non-zero morphism $\Hom_\C(-,Y)\to S$ to a simple object is a
projective cover if and only if $\End_\C(Y)$ is local.  In that case
the exactness of the sequence means that the image of
$\Hom_\C(-,X)\to\Hom_\C(-,Y)$ equals $\Rad_\C(-,Y)$; it is therefore
equivalent to the fact that $X\to Y$ is right almost split. The
canonical morphism from $\Hom_\C(-,X)$ to the image of
$\Hom_\C(-,X)\to\Hom_\C(-,Y)$ is a projective cover if and only if
$X\to Y$ is right minimal.
\end{proof}

Almost split morphisms and irreducible
morphisms are related as follows.

\begin{lem}[{\cite[Corollary~3.4]{Ba}}]\label{le:AR}
Let $X\to Y$ be a minimal right almost split morphism in $\C$ and let
$X=X_1^{n_1}\amalg\ldots\amalg X_r^{n_r}$ be a decomposition into
indecomposable objects such that the $X_i$ are pairwise
non-isomorphic.  Given an indecomposable object $X'$, one has
$\Irr_\C(X',Y)\neq 0$ if and only if $X'\cong X_i$ for some $i$. In
that case $n_i$ equals the length of $\Irr_\C(X',Y)$ over $\De(X')$.
\end{lem}
\begin{proof}
The morphism $X\to Y$ induces a minimal projective presentation
\[\Hom_\C(-,X)\lto\Hom_\C(-,Y)\lto S_Y\lto 0\]
by Lemma~\ref{le:cover}, and therefore a projective
cover \[\pi\colon\Hom_\C(-,X)\lto\Rad_\C(-,Y).\] This morphism induces
an isomorphism
\[\Hom_\C(-,X)/\Rad_\C(-,X)\stackrel{\sim}\lto \Rad_\C(-,Y)/\Rad^2_\C(-,Y)\]
since $\Ker\pi\subseteq\Rad_\C(-,X)$.  On the other hand, the
decomposition of $X$ implies
\[\Hom_\C(-,X)/\Rad_\C(-,X)\cong S_{X_1}^{n_1}\amalg\ldots\amalg S_{X_r}^{n_r}.\] 
It remains to observe that $S_{X_i}(X')\neq 0$ if and only if $X'\cong
X_i$.
\end{proof}

\subsection*{Almost split sequences}
The following definition of an almost split sequence is taken from Liu
\cite{Li}; it covers the original definition of Auslander and Reiten
for abelian categories \cite{AR1975}, but also Happel's definition of
an Auslander--Reiten triangle in a triangulated category
\cite{Ha1987}.

A sequence of morphisms $X\xto{\a} Y\xto{\b} Z$ in $\C$ is called
\emph{almost split} if
\begin{enumerate}
\item $\a$ is minimal left almost split and a weak kernel of $\b$, 
\item $\b$ is minimal right almost split and a weak cokernel of $\a$, and
\item $Y\neq 0$.
\end{enumerate}

The end terms $X$ and $Z$ determine each other up to isomophism, and
we write $X=\t Z$ and $Z=\t^{-1}X$. One calls $\t Z$ the
\emph{Auslander--Reiten translate} of $Z$.

\begin{lem}\label{le:ass}
A sequence of morphisms $X\to Y\to Z$ in $\C$ is almost split if and
only if it induces two minimal projective presentations
\begin{gather*}\Hom_\C(-,X)\to\Hom_\C(-,Y)\to\Hom_\C(-,Z)\to S_Z\to 0\\
\Hom_\C(Z,-)\to\Hom_\C(Y,-)\to\Hom_\C(X,-)\to S^X\to 0
\end{gather*}
where we use the notation
\[S^X=\Hom_\C(X,-)/\Rad_\C(X,-).\]
\end{lem}
\begin{proof}
Apply Lemma~\ref{le:cover}.
\end{proof}

The Auslander--Reiten translate is functorial in the following sense.

\begin{lem}\label{le:tau}
Let $X,Y$ be indecomposable objects in $\C$ and suppose their
Auslander--Reiten translates are defined. Then $\De(X)\cong\De(\t X)$
and $\De(Y)\cong\De(\t Y)$ as division rings. Using these
isomorphisms, we have
\[\Hom_{\De(X)}(\Irr_\C(X,Y),\De(X))\cong\Hom_{\De(\t Y)}(\Irr_\C(\t X,\t Y),\De(\t Y))\]
as bimodules over $\De(X)$ and $\De(Y)$.
\end{lem}
\begin{proof}
We use the \emph{stable category} of finitely presented $\C$-modules
which we denote by $\umod\C$. The objects are the $\C$-modules $M$
that admit a presentation
\[\Hom_\C(-,X) \stackrel{(-,\p)}\lto\Hom_\C(-,Y)\lto M\lto 0\tag{$*$}\]
and the morphisms are all $\C$-linear morphisms modulo the subgroup of
morphisms that factor through a projective $\C$-module. There are two functors
\[\Om\colon\umod\C\lto\umod\C\quad\text{and}\quad\Tr\colon\umod\C\lto\umod\C^\op\]
taking a module to its \emph{syzygy} and its \emph{transpose},
respectively. Both functors are defined for a module $M$ with
presentation $(*)$ via  exact sequences as follows:
\begin{gather*}0\lto\Om M\lto\Hom_\C(-,Y)\lto M\lto 0\\
\Hom_\C(Y,-)\stackrel{(\p,-)}\lto\Hom_\C(X,-)\lto \Tr M\lto 0
\end{gather*}
Note that the transpose yields an equivalence $(\umod\C)^\op\xto{\sim}\umod\C^\op$.

Using the minimal projective presentations of $S_X$ and $S^{\t X}$
from Lemma~\ref{le:ass}, one gets isomorphisms
\[\Om S_X\cong\Tr S^{\t X}\qquad\text{and}\qquad\Om S^{\t X}\cong\Tr S_X.\]
These yield mutually inverse maps
\begin{multline*}
\De(X)\xto{\sim}\uEnd_\C(S_X)\to \uEnd_\C(\Om
S_X)\xto{\sim}\\ \uEnd_\C(\Tr S^{\t X})\xto{\sim}
\uEnd_{\C^\op}(S^{\t X})^\op \xto{\sim}\De(\t X)
\end{multline*}
and 
\begin{multline*}
\De(\t X)\xto{\sim}\uEnd_{\C^\op}(S^{\t X})^\op
\to\uEnd_{\C^\op}(\Om S^{\t X})^\op
\xto{\sim}\\ \uEnd_{\C^\op}(\Tr S_{ X})^\op
\xto{\sim} \uEnd_{\C}(S_{ X})\xto{\sim}\De(X).
\end{multline*}

Next we apply Lemma~\ref{le:irr} and obtain
\begin{align*}
\Hom_{\De(X)}(\Irr_\C(X,Y),\De(X))
&\cong\Ext^1_\C(S_Y,S_X)\\
&\cong\uHom_\C(\Om S_Y,S_X)\\
&\cong\uHom_\C(\Tr S^{\t Y},S_X)\\
&\cong\uHom_{\C^\op}(\Tr S_X,S^{\t Y})\\
&\cong\uHom_{\C^\op}(\Om S^{\t X},S^{\t Y})\\
&\cong\Ext^1_{\C^\op}(S^{\t X},S^{\t Y})\\
&\cong \Hom_{\De(\t X)}(\Irr_{\C^\op}(\t Y,\t X),\De(\t X))\\
&\cong \Hom_{\De(\t X)}(\Irr_{\C}(\t X,\t Y),\De(\t X)).
\end{align*}
The second isomorphism requires an extra argument, and the same is
used for the sixth. Let $\Om S_Y=\Rad_\C(-,Y)$, and observe that this
has no non-zero projective direct summand by the minimality of the
presentations in Lemma~\ref{le:ass}.  Then we have
\[\Ext^1_\C(S_Y,S_X)\cong\Hom_\C(\Om S_Y,S_X)\cong\uHom_\C(\Om S_Y,S_X).\]
The first isomorphism is from the proof of Lemma~\ref{le:irr}.  The
second follows from the fact that any non-zero morphism $\Om S_Y\to
S_X$ factoring through a projective factors through the projective
cover $\Hom_\C(-,X)\to S_X$. This means that $\Hom_\C(-,X)$ is a
direct summand of $\Om S_Y$, which has been excluded before.
\end{proof}

\subsection*{The Auslander--Reiten quiver}
The \emph{Auslander--Reiten quiver} of $\C$ is defined as follows. The
set of isomorphism classes of indecomposable objects in $\C$ form the
\emph{vertices}, and there is an \emph{arrow} $X\to Y$ if
$\Irr_\C(X,Y)\neq 0$. It is often convenient to identify an
indecomposable object with its isomorphism class.

The Auslander--Reiten quiver carries a \emph{valuation} which assigns
to each arrow $X\to Y$ the pair $(\d_{X,Y},\d'_{X,Y})$, where
\[\d_{X,Y}=\ell_{\De(X)}(\Irr_\C(X,Y))\quad\text{and}\quad
\d'_{X,Y}=\ell_{\De(Y)}(\Irr_\C(X,Y)).\]
Here, $\ell_A(M)$ denotes the length of an $A$-module $M$.

\begin{lem}\label{le:val}
Let $X,Y$ be indecomposable objects in $\C$ and suppose that $\t Y$ is
defined.  Then $\d'_{\t Y,X}=\d_{X,Y}$.
\end{lem}
\begin{proof}
We have an almost split sequence $\t Y\to \bar X\to Y$ and $\d_{X,Y}$
counts the multiplicity of $X$ in a decomposition of $\bar X$ by Lemma~\ref{le:AR},
which equals $\d'_{\t Y,X}$ by the dual of Lemma~\ref{le:AR}.
\end{proof}

\begin{lem}\label{le:val'}
Let $k$ be a commutative ring and suppose that $\C$ is $k$-linear with
all morphism spaces of finite length over $k$.  Let $X,Y$ be
indecomposable objects in $\C$ and suppose that $\t Y$ is defined. Then
$\d_{\t Y,X}=\d'_{X,Y}$.
\end{lem}
\begin{proof}
Using the identity $\d'_{\t Y,X}=\d_{X,Y}$ from Lemma~\ref{le:val} and
the isomorphism $\De(Y)\cong\De(\t Y)$ from Lemma~\ref{le:tau}, one
computes
\begin{align*}
\ell_{\De(Y)}(\Irr_\C(X,Y))&=\ell_k(\Irr_\C(X,Y))\cdot\ell_k(\De(Y))^{-1}\\
&=\ell_{\De(X)}(\Irr_\C(X,Y))\cdot\ell_k(\De(X))^{}\cdot\ell_k(\De(Y))^{-1}\\
&=\ell_{\De(X)}(\Irr_\C(\t Y,X))\cdot\ell_k(\De(X))^{}\cdot\ell_k(\De(Y))^{-1}\\
&=\ell_k(\Irr_\C(\t Y,X))\cdot\ell_k(\De(\t Y))^{-1}\\
&=\ell_{\De(\t Y)}(\Irr_\C(\t Y,X)).\qedhere
\end{align*}
\end{proof}

\subsection*{The repetition}
Let $\Ga$ be a quiver without loops. Then a new quiver $\bbZ\Ga$ is
defined as follows. The set of vertices is $\bbZ\times\Ga_0$. For each
arrow $x\to y$ in $\Ga$ and each $n\in\bbZ$ there is an arrow
$(n,x)\to (n,y)$ and an arrow $(n-1,y)\to (n,x)$ in $\bbZ\Ga$.  The
quiver $\bbZ\Ga$ is a translation quiver \cite{Ri1980} with translation
$\t$ defined by $\t (n,x)=(n-1,x)$ for each vertex $(n,x)$.

\end{appendix}

\end{document}